[1994/12/01]
\documentclass{ijmart-mod}
\chardef\bslash=`\\ 

\usepackage{graphicx}
\usepackage[breaklinks=true]{hyperref}

\newtheorem{thm}{Theorem}[section]
\newtheorem{cor}[thm]{Corollary}
\newtheorem{lem}[thm]{Lemma}

\theoremstyle{definition}
\newtheorem{defn}[thm]{Definition}

\newtheorem{qtn}[thm]{Question}

\theoremstyle{remark}

\newcommand{\eval}[2][\right]{\relax
  \ifx#1\right\relax \left.\fi#2#1\rvert}

\begin{document}
\title{Elementary moves on lattice polytopes}

\author{Julien David}
\address{LIPN, Universit{\'e} Paris 13, Villetaneuse, France}
\email{julien.david@lipn.univ-paris13.fr} 

\author{Lionel Pournin}
\address{LIPN, Universit{\'e} Paris 13, Villetaneuse, France}
\email{lionel.pournin@univ-paris13.fr} 

\author{Rado Rakotonarivo}
\address{LIPN, Universit{\'e} Paris 13, Villetaneuse, France}
\email{rakotonarivo@lipn.univ-paris13.fr} 

\begin{abstract}
We introduce a graph structure 
 on Euclidean polytopes. The vertices of this graph are the $d$-dimensional polytopes contained in $\mathbb{R}^d$ and its edges connect any two polytopes that can be obtained from one another by either inserting or deleting a vertex, while keeping their vertex sets otherwise unaffected. We prove several results on the connectivity of this graph, and on a number of its subgraphs. We are especially interested in several families of subgraphs induced by lattice polytopes, such as the subgraphs induced by the lattice polytopes with $n$ or $n+1$ vertices, that turn out to exhibit intriguing properties.
\end{abstract}
\maketitle

\section{Introduction}\label{sec.DPR.0}

In this paper, we introduce a graph structure on the Euclidean polytopes. The vertices of this graph are the $d$-dimensional polytopes contained in $\mathbb{R}^d$ and its edges connect two polytopes when they can be obtained from one another by a transforming move, that we want to keep as elementary as possible. Here, by \emph{elementary}, we mean that these moves should preserve as much as possible the combinatorics of a polytope's boundary. We will consider two types of moves, an insertion move and a deletion move. If $x$ is a point in $\mathbb{R}^d\mathord{\setminus}P$, an insertion move will transform $P$ into the convex hull of $P\cup\{x\}$. Note that $x$ is then necessarily a vertex of the resulting polytope. If $v$ is a vertex of $P$, a deletion move will transform $P$ into the convex hull of $\mathcal{V}\mathord{\setminus}\{v\}$, where $\mathcal{V}$ is the vertex set of $P$. Without any other requirement, these moves cannot be considered elementary in the above sense, as they can alter the combinatorics of the boundary complex of $P$ significantly. Indeed, the convex hull of $P\cup\{x\}$ can have fewer (and possibly many less) vertices that $P$ itself. This happens, for instance, when the convex hull of $P\cup\{x\}$ contains at least two vertices of $P$ in its relative interior. An undesirable consequence is that deletion moves would then not be the inverse of insertion moves as $P$ would not always be recovered by deleting $x$ from the convex hull of $P\cup\{x\}$. A natural way to solve this issue consists in allowing an insertion move only when all the vertices of $P$ remain vertices of the polytope resulting from that insertion.

\begin{defn}
Consider a $d$-dimensional polytope $P$ contained in $\mathbb{R}^d$ and denote by $\mathcal{V}$ its set of vertices. A point $x\in\mathbb{R}^d$ can be inserted in $P$ if the convex hull of $P\cup\{x\}$ admits $\mathcal{V}\cup\{x\}$ as its vertex set. A vertex $v\in\mathcal{V}$ can be deleted from $P$ when the convex hull of $\mathcal{V}\mathord{\setminus}\{v\}$ is $d$-dimensional.
\end{defn}

By this definition, deletion moves and insertion moves are now the inverse of one another. Recall that all the polytopes we consider here are full-dimensional, making the requirement that deletion moves do not decrease the dimension of a polytope necessary. In particular, a vertex $v$ of a polytope $P$ can be deleted from $P$ if and only if $P$ is not a pyramid with apex $v$ over a $(d-1)$-dimensional polytope. Consider the graph whose vertices are the $d$-dimensional polytopes contained in $\mathbb{R}^d$ and whose edges connect two polytopes that can be obtained from one another by an insertion move (or a deletion move). This graph, which we will refer to as $\Gamma(d)$ here, has an uncountable number of vertices and, as soon as $d\geq2$, its vertices all have uncountable degree. Indeed, consider an arbitrary point $x$ in the boundary of a polytope $P$, distinct from any vertex of $P$. One can insert in $P$ any point outside of $P$ that is close enough to $x$.

The purpose of this paper is to investigate the connectivity of a number of subgraphs of $\Gamma(d)$. Throughout the paper, it will always be assumed that $d$ is at least $2$. Our first main result, proven in Section \ref{Sec.DPR.1}, deals with the subgraphs induced in $\Gamma(d)$ by the polytopes with $n$ or $n+1$ vertices.

\begin{thm}\label{thm.main.1}
The polytopes with $n$ or $n+1$ vertices induce a connected subgraph of $\Gamma(d)$ whose diameter is at least $4n-d$ and at most $6n-4$.
\end{thm}

The connectedness of $\Gamma(d)$ itself will be obtained as a consequence of Theorem~\ref{thm.main.1}. Note that $\Gamma(d)$ provides a metric on the set of $d$-dimensional polytopes, in a very different spirit than, for instance the Gromov-Hausdorff distance \cite{Gromov1981}: instead of measuring how far two bodies are from being isometric, we measure how long it takes to build them from one another with operations that affect as little as possible the combinatorics of their boundary. 

The two families of graphs we mostly focus on are the subgraph $\Lambda(d)$ induced in $\Gamma(d)$ by the lattice polytopes and the subgraph $\Lambda(d,k)$ induced in $\Lambda(d)$ by the polytopes contained in the hypercube $[0,k]^d$, where $k$ is a positive integer. Here, by a \emph{lattice polytope} we mean a polytope whose vertices belong to the lattice $\mathbb{Z}^d$. These polytopes pop up in many places in the mathematical literature as, for instance in combinatorial optimization \cite{Naddef1989,NillZiegler2011,Sebo1999}, in discrete geometry \cite{BaranyPach1992,BrionVergne1997,LagariasZiegler1991}, or in connection with toric varieties \cite{BogartHaaseHeringLorenzNillPaffenholzRoteSantosSchenck2015,DickensteinDiRoccoPiene2009}. Note that, in the case of lattice polytopes, alternative deletion and insertion moves have been considered, that amount to change the number of lattice points contained in a polytope by exactly one \cite{BlancoSantos2018,BlancoSantos2019,BrunsGubeladzeMichalek2016}. They can be used to enumerate the lattice polytopes that contain a fixed number of lattice points \cite{BlancoSantos2018,BlancoSantos2019} but, in contrast to our moves, they can affect the combinatorics of a polytope's boundary in an arbitrary way. Observe that $\Lambda(d)$ is a highly non-regular graph: it admits both vertices with finite degree and vertices with infinite (but countable) degree. In particular, $\Lambda(d)$ gathers in a coherent metric structure polytopes whose boundaries exhibit dramatically different behaviors regarding the ambient lattice. For instance, there are lattice polytopes, like the cube $[0,1]^d$, in which no lattice point can be inserted, while for the lattice simplices, no deletion move is possible. The graph $\Lambda(d,k)$ is particularly relevant to the study of the lattice polytopes contained in $[0,k]^d$, that have attracted significant attention \cite{AcketaZunic1995,BaranyLarman1998,DelPiaMichini2016,DezaManoussakisOnn2018,DezaPournin2018,KleinschmidtOnn1992}. Our second main result deals with a subgraph of $\Lambda(d,k)$.

\begin{thm}\label{thm.main.2}
The subgraph induced in $\Lambda(d,k)$ by the simplices and the polytopes with $d+2$ vertices is connected.
\end{thm}

This theorem will be proven in Sections \ref{Sec.DPR.2} and \ref{Sec.DPR.3} along with a number of its consequences. For instance, it is proven in Section \ref{Sec.DPR.2} that some lattice point in $[0,k]^d$ can always be inserted in a $d$-dimensional lattice simplex contained in $[0,k]^d$. As we shall see, the proof of this seemingly straightforward statement alone turns out to be surprisingly involved. The connectedness of $\Lambda(d)$ and $\Lambda(d,k)$ will both be obtained from Theorem \ref{thm.main.2} in Section \ref{Sec.DPR.3}. Observe that the latter connectedness result allows for the definition of a Markov chain whose states are the $d$-dimensional lattice polytopes contained in $[0,k]^d$, and whose stationary distribution is uniform \cite{DavidPourninRakotonarivo2018}.

In Section \ref{Sec.DPR.4}, we will study the number of lattice points that can be inserted in, or removed from a lattice polytope. We will describe a family of $d$-dimensional lattice polytopes contained in the hypercube $[0,k]^d$, where $d$ and $k$ can grow arbitrarily large, such that every lattice point in $[0,k]^d$ can be either inserted or deleted. These polytopes belong to the broader family of the empty lattice polytopes, that is widely studied and is interesting in its own right \cite{BaranyKantor2000,BarileBernardiBorisovKantor2011,DezaOnn1995,HaaseZiegler2000,Kantor1999,SantosValino2018,Sebo1999}. We will also exhibit lattice polytopes with arbitrarily large dimension and number of vertices such that no insertion of a lattice point is possible. As an immediate consequence, the subgraph induced in $\Lambda(d)$ by the polytopes with $n$ or $n+1$ vertices is not always connected.

In particular we obtain the following in Section \ref{Sec.DPR.4}.
\begin{thm}\label{thm.main.3}
The subgraph induced in $\Lambda(2)$ by the polygons with $n$ or $n+1$ vertices is disconnected when $n$ is distinct from $3$ and $5$.
\end{thm}

It is a consequence of Theorem \ref{thm.main.2} that triangles and quadrilaterals induce a connected subgraph of $\Lambda(2)$, which settles the first exception in the statement of Theorem~\ref{thm.main.3}. We settle the other exception in Section \ref{Sec.DPR.5} as follows.
 
 \begin{thm}\label{thm.main.4}
Pentagons and hexagons induce a connected subgraph of $\Lambda(2)$.
\end{thm}

In order to prove Theorem \ref{thm.main.4}, we will study a particular connected component of the subgraph induced in $\Lambda(2)$ by the polygons with $n$ or $n+1$ vertices. The proof of Theorem \ref{thm.main.4} consists in showing that this connected component is the whole subgraph when $n$ is equal to $5$. We conclude the article in Section~\ref{Sec.DPR.6} by asking a number of questions. Part of these questions arise naturally from our results, and in particular from the intriguing behavior of the subgraphs induced in $\Lambda(d)$ by the polytopes with $n$ or $n+1$ vertices. We will also mention a  number of other subgraphs of $\Gamma(d)$, whose study may be interesting.

\section{The connectivity of $\Gamma(d)$}\label{Sec.DPR.1}

In this section we investigate the connectedness of $\Gamma(d)$ itself and of its subgraphs induced by the polytopes with $n$ and $n+1$ vertices, where $n\geq{d+1}$. We also obtain precise bounds on the diameter of these subgraphs. Note that from now on, it is always assumed that $d$ is not less than $2$.

Consider a $d$-dimensional polytope $P$ contained in $\mathbb{R}^d$. We will denote by $\mathrm{aff}(F)$ the affine hull of a face $F$ of $P$. If $F$ is a facet, then $\mathrm{aff}(F)$ is a hyperplane of $\mathbb{R}^d$ and we denote by $H_F^-(P)$ the closed half-space of $\mathbb{R}^d$ bounded by $\mathrm{aff}(F)$ such that $P\cap{H_F^-(P)}=F$. For any vertex $v$ of $P$, the set
\begin{equation}
C_v(P)=\bigcap_{F\in\mathcal{F}}H_F^-(P)\mbox{,}
\end{equation}
where $\mathcal{F}$ is the set of the facets of $P$ incident to $v$, is a $d$-dimensional polyhedral cone pointed at $v$. This cone is exactly the set of the points $x\in\mathbb{R}^d$ such that the convex hull of $P\cup\{x\}$ does not admit $v$ as a vertex. By this remark, we immediately obtain the following lemma.

\begin{lem}\label{Lem.DPR.1.A}
Consider a $d$-dimensional polytope $P$ contained in $\mathbb{R}^d$. A point $x$ of $\mathbb{R}^d$ can be inserted in $P$ if and only if it does not belong to $P$ and, for every vertex $v$ of $P$, it does not belong to $C_v(P)$.
\end{lem}

Lemma \ref{Lem.DPR.1.A} will be instrumental to prove the connectedness of a number of subgraphs of $\Gamma(d)$ in this section and the next. We will also make use of the following technical lemma that describes how the cones $C_v(P)$ are placed relatively to the supporting hyperplanes of the faces of a polytope $P$.

\begin{lem}\label{Lem.DPR.1.B}
Consider a proper face $F$ of a $d$-dimensional polytope $P$ contained in $\mathbb{R}^d$. Let $H$ be a hyperplane such that $F=P\cap{H}$, and $H^-$ the half-space of $\mathbb{R}^d$ bounded by $H$ and such that $F=P\cap{H^-}$. For any vertex $v$ of $F$, $C_v$ is a subset of $H^-$ and $C_v\cap{H}$ is a subset of $\mathrm{aff}(F)$.
\end{lem}

\begin{proof}
Consider the intersection
$$
K=\bigcap_{G\in\mathcal{G}}H_G^-(P)\mbox{,}
$$
where $\mathcal{G}$ is the set of all the facets of $P$ incident to $F$. First observe that $K$ is a subset of $H^-$. In addition $K\cap{H}$ is precisely the affine hull of $F$. Now consider a vertex $v$ of $F$. Since $C_v(P)$ is a subset of $K$, the result follows.
\end{proof}

We need to state another elementary result. By the following lemma, all the vertices of a $d$-dimensional polytope can be deleted from it except at most $d+1$ of them. In particular, the only polytopes whose none of the vertices can be deleted from are the simplices. The proof of this result, by induction on the dimension, is straightforward and will be omitted.

\begin{lem}\label{Lem.DPR.1.BB}
If $P$ is a $d$-dimensional polytope, then $P$ cannot be a pyramid over more than $d+1$ of its facets.
\end{lem}

We now investigate the connectedness of the subgraph induced in $\Gamma(d)$ by the polytopes with $n$ or $n+1$ vertices. The following lemma deals with a special case that will pop up several times thereafter. We say that a subset $\mathcal{A}$ of $\mathbb{R}^d$ is in \emph{convex position} if any finite subset of $\mathcal{A}$ is the vertex set of a polytope.

\begin{lem}\label{Lem.DPR.1.BC}
Let $\mathcal{A}$ be a $d$-dimensional subset of $\mathbb{R}^d$ in convex position. For any $n\geq{d+1}$, the polytopes with $n$ or $n+1$ vertices whose vertex set is a subset of $\mathcal{A}$ induce a connected subgraph of $\Gamma(d)$ of diameter at most $2n+2$. Moreover, two polytopes with $n$ vertices have distance at most $2n$ in this subgraph.
\end{lem}
\begin{proof}
Let $P$ and $Q$ be two polytopes with $n$ or $n+1$ vertices such that the vertex sets of $P$ and $Q$ are subsets of $\mathcal{A}$. By Lemma \ref{Lem.DPR.1.BB}, we can assume without loss of generality that both $P$ and $Q$ have exactly $n$ vertices. Since $\mathcal{A}$ is in convex position, the vertices of the convex hull of $P\cup{Q}$ are exactly the vertices of $P$ and the vertices of $Q$. As a consequence, any vertex of $Q$ that is not already a vertex of $P$ can be inserted in $P$. After such an insertion, Lemma \ref{Lem.DPR.1.BB} makes sure that one can, in turn, delete a vertex distinct from the inserted point. The polytope resulting from this sequence of two moves shares at least one more vertex with $Q$ than $P$. Repeating this process therefore builds a path between $P$ and $Q$ in the considered subgraph of $\Gamma(d)$. The length of this path is at most the sum of the number of vertices of $P$ with the number of vertices of $Q$. Taking into account the two deletion moves that have possibly been performed initially to build $P$ and $Q$ from polytopes with $n+1$ vertices, we obtain the desired bound on the diameter of that subgraph.
\end{proof}

Lemma \ref{Lem.DPR.1.BC} is instrumental already in the proof of the following theorem.

\begin{thm}\label{Thm.DPR.1.A}
For any $n\geq{d+1}$, the polytopes with $n$ or $n+1$ vertices induce a connected subgraph of $\Gamma(d)$ of diameter at most $6n-2$. Moreover, two polytopes with $n$ vertices are distant of at most $6n-4$ in this subgraph.
\end{thm}
\begin{proof}
Consider two $d$-dimensional polytopes $P$ and $Q$ contained in $\mathbb{R}^d$ both with $n$ or $n+1$ vertices. Since it is always possible to delete some vertex from a polytope with more than $d+1$ vertices, we can assume without loss of generality that both $P$ and $Q$ have $n$ vertices. Denote
$$
\gamma=\min\{x_1:x\in{P\cup{Q}}\}\mbox{.}
$$

We will assume that some point $x$ of $P$ satisfies $x_1=\gamma$, which can be done by exchanging $P$ and $Q$, if need be. Let $M$ denote the hyperplane made up of the points $x$ such that $x_1=\gamma$. The intersection of $P$ and $M$ is a non-empty face of $P$. This face, that we denote by $E$, is sketched on the left of Fig. \ref{Fig.DPR.1.Z}. As illustrated in the figure, if $E$ is not a facet of $P$, it is always possible to insert in $P$ some point in $M$ that does not belong to the affine hull of $E$ (for instance, the point colored green on the left of Fig. \ref{Fig.DPR.1.Z}). Indeed, by Lemma \ref{Lem.DPR.1.B}, for any vertex $v$ of $E$, the intersection $C_v(P)\cap{M}$ is a subset of $\mathrm{aff}(E)$. Moreover, for any vertex $v$ of $P$ that is not incident to $E$, $C_v(P)\cap{M}$ is necessarily disjoint from $E$. Since $C_v(P)\cap{M}$ and $E$ both are closed sets, it follows from Lemma~\ref{Lem.DPR.1.A} that any point $x$ in $M$ that does not belong to $\mathrm{aff}(E)$ can be inserted in $P$, provided it is close enough to $E$. After $x$ has been inserted in $P$, any vertex of $P$ that is not incident to $E$ can be deleted: if such a vertex were deleted from $P$, then the resulting polytope would be at least $(d-1)$-dimensional. Therefore, deleting it after $x$ has been inserted in $P$ results in a $d$-dimensional polytope because $x$ is not contained in the affine hull of $E$. Repeating this procedure, one can transform $P$ into a polytope $P'$ such that $P'\cap{M}$ is a facet of $P'$, using at most $2d-2$ moves (at most $d-1$ insertion moves, each followed by a deletion move). We denote $E'=P'\cap{M}$. Now call
$$
\delta=\max\{x_1:x\in{Q}\}\mbox{,}
$$ 
and let $N$ be the hyperplane made up of the points $x$ such that $x_1=\delta$. As above, if $Q\cap{N}$ is not a facet of $Q$, we can perform a sequence of at most $d-1$ insertion moves on $Q$, that insert points in $N$, each followed by a deletion move, in order to obtain a polytope $Q'$ such that $Q'\cap{N}$ is a facet of $Q'$. In the remainder of the proof, we denote that facet by $F'$. Note that, the sketch on the left of Fig. \ref{Fig.DPR.1.Z} depicts the case when $Q'=Q$.

Now, consider a point in $H_{E'}^-(P')\mathord{\setminus}\mathrm{aff}(E')$ whose orthogonal projection on $\mathrm{aff}(E')$ is contained in the relative interior of $E'$ as, for instance, the points colored green next to $E'$ on the right of Fig. \ref{Fig.DPR.1.Z}. Observe that this point can be inserted in $P'$ provided it is close enough to $E'$. We will perform a sequence of such insertion moves. Note that, while the first insertion move is in the neighborhood of $E'$, the next insertion moves will be made in the neighborhood of one of the facets introduced by the preceding insertion, in such a way that the orthogonal projection on $\mathrm{aff}(E')$ of each inserted point is contained in the relative interior of $E'$. After each insertion move, a deletion move will be performed on a vertex of $P'$ that is not incident to $E$ (note that any such vertex, colored red on the right of Fig. \ref{Fig.DPR.1.Z}, can be deleted). Doing so, we can transform $P'$ into a polytope $P''$ that admits $E'$ as a facet and whose vertices not incident to $E'$ can be placed arbitrarily close to $E'$. In particular we can require that, for any facet $G$ of $P''$ other than $E'$, $H_G^-(P'')$ and $F'$ are disjoint.
\begin{figure}
\begin{center}
\includegraphics{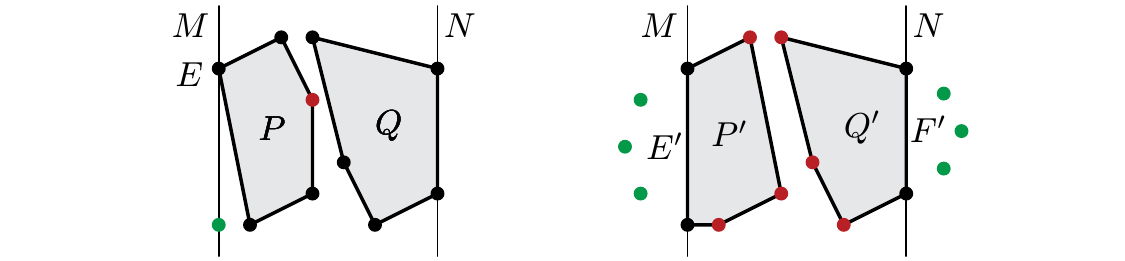}
\caption{The moves that transform $P$ into $P''$ and $Q$ into $Q''$.}
\label{Fig.DPR.1.Z}
\end{center}
\end{figure}
Similarly, we can find a sequence of insertion moves, each followed by a deletion move that transform $Q'$ into a polytope $Q''$ that admits $F'$ as a facet and such that, for any other facet $G$ of $Q''$, $H_G^-(Q'')$ and $E'$ are disjoint.

By that construction, all the vertices of $P''$ and all the vertices of $Q''$ are vertices of the convex hull of $P''\cup{Q''}$. Hence, by Lemma \ref{Lem.DPR.1.BC}, one can transform $P''$ into $Q''$ by a sequence of at most $2n$ moves such that each insertion move is followed by a deletion move. We have done at most $2n-2d$ moves to transform $P'$ into $P''$ or $Q'$ into $Q''$, and at most $2n$ moves to transform $P''$ into $Q''$. Taking into account the two deletion moves that have possibly been performed initially to build $P$ and $Q$ from polytopes with $n+1$ vertices, we obtain the desired upper bound on the diameter of the subgraph induced in $\Gamma(d)$ by the polytopes with $n$ or $n+1$ vertices.
\end{proof}

A consequence of Theorem \ref{Thm.DPR.1.A} is that it is always possible to transform two polytopes into one another using a sequence of elementary moves.

\begin{cor}\label{Cor.DPR.1.A}
$\Gamma(d)$ is connected.
\end{cor}
\begin{proof}
Let $P$ and $Q$ be two $d$-dimensional polytopes contained in $\mathbb{R}^d$. Say that $P$ has $n$ vertices and $Q$ has $m$ vertices. We can assume without loss of generality that $n\leq{m}$. By Lemma \ref{Lem.DPR.1.BB}, it is always possible to delete some vertex from a polytope with more than $d+1$ vertices, there is a (possibly empty) sequence of deletion moves that transform $Q$ into a $d$-dimensional polytope with $n$ vertices. The result then follows from Theorem \ref{Thm.DPR.1.A}.
\end{proof}

In the remainder of the section, we look at the distance between two polytopes in $\Gamma(d)$. According to Theorem \ref{Thm.DPR.1.A}, the diameter of the subgraph induced in $\Gamma(d)$ by polytopes with $n$ or $n+1$ vertices is at most $6n-2$. This upper bound is linear in the number $n$ of vertices of the considered polytopes, but it is independent on the dimension, which may be surprising. We obtain a lower bound on that diameter that is reasonably close to our upper bound.

\begin{lem}\label{Lem.DPR.1.C}
For any $n\geq{d+1}$, the subgraph of $\Gamma(d)$ induced by the polytopes with $n$ or $n+1$ vertices has diameter at least $4n-2d$.
\end{lem}
\begin{proof}
Consider the two polygons $P$ and $Q$ sketched on the left of Fig. \ref{Fig.DPR.1.A}. We will assume that each of these polygons has $n-d+2$ vertices. As can be seen on the figure, $P$ and $Q$ are placed in such a way that for any vertex $v$ of $P$ distinct from the vertices of its longest edge, $Q$ is a subset of the cone $C_v(P)$. The intersection of all these cones is shown as a striped surface in the figure. Observe that, for any $d\geq3$, we can build a $d$-dimensional polytope by considering a pyramid over $P$, and then a pyramid over that pyramid and so on. We will call $P'$ the resulting $d$-dimensional polytope, and $Q'$ the $d$-dimensional polytope obtained using the same procedure but starting from $Q$ instead of $P$.
\begin{figure}
\begin{center}
\includegraphics{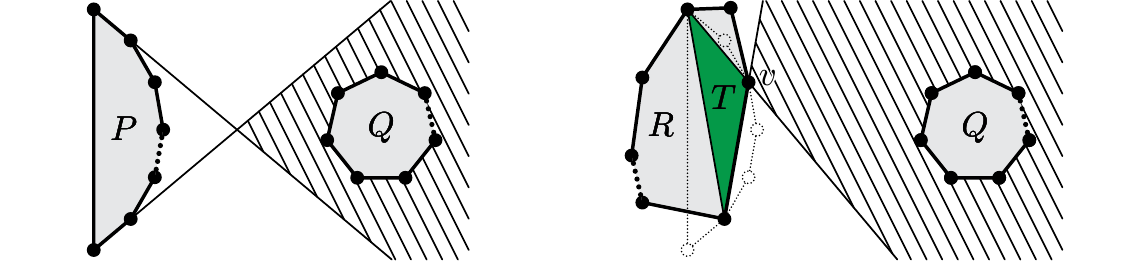}
\caption{Polygons $P$, $Q$, $R$, and the triangle $T$.}
\label{Fig.DPR.1.A}
\end{center}
\end{figure}
By construction, both $P'$ and $Q'$ have $n$ vertices. We require, which can be done without loss of generality, that $P'$ and $Q'$ do not share a vertex.

Now consider a sequence of insertion moves, each followed by a deletion move that transform $P'$ into $Q'$. Consider the first move in that sequence that introduces a vertex of $Q$. This move is performed on a polytope $R'$. We claim that, when this move occurs, all except maybe two vertices of $P$ have been deleted from $P'$. Indeed, otherwise, the intersection of $R'$ with the plane that contains $P$ and $Q$ is a polygon $R$ that shares at least three vertices with $P$. As can be seen on the right of Fig. \ref{Fig.DPR.1.A}, where the trace of $P$ is sketched with dotted lines, these three vertices form a triangle $T$, depicted in green.  Note that the vertices of $R$ are possibly not all vertices of $R'$: some of them may be the intersection of a higher dimensional face of $R'$ with the plane that contains $P$ and $Q$. However, $R$ and $R'$ necessarily share the three vertices of $T$. Now observe that $T$ admits at least one vertex $v$ such that $Q$ is a subset of the cone $C_v(T)$. This cone, shown as a striped surface in the figure, is in turn a subset of $C_v(R)$. Since $C_v(R)$ is contained in $C_v(R')$, we obtain the inclusion $Q\subset{C_v(R')}$. In particular, no vertex of $Q$ can be inserted in $R'$.

Hence, there must have been at least $n-d$ insertion moves, each followed by a deletion move before $R'$ is reached from $P'$. After that, all the vertices of $Q$ still have to be introduced, which requires at least $n-d+2$ insertion moves and $n-d+2$ deletion moves. Since $P'$ and $Q'$ do not share a vertex, we further need to perform at least $d-2$ insertions and $d-2$ deletions to displace the vertices of $P'$ that are not incident to $P$. As a consequence, transforming $P'$ into $Q'$ requires at least $4n-2d$ moves.
\end{proof}

Note that Theorem \ref{thm.main.1} is obtained by combining Theorem \ref{Thm.DPR.1.A} and Lemma~\ref{Lem.DPR.1.C}. A consequence of these results is that simplices play a central role in $\Gamma(d)$, in the sense that connecting two polytopes with $n$ vertices in $\Gamma(d)$ via a simplex can be much shorter than with any path visiting only polytopes with $n$ or $n+1$ vertices. In fact, according to Lemma \ref{Lem.DPR.1.D}, paths via simplices can be at least half as short when $d$ is fixed and $n$ grows large.

\begin{lem}\label{Lem.DPR.1.D}
The distance in $\Gamma(d)$ between a polytopes with $n$ vertices and a polytope with $m$ vertices is at most $n+m+4d$.
\end{lem}
\begin{proof}
By Lemma \ref{Lem.DPR.1.BB}, one can always transform a polytope into a simplex by performing a sequence of deletions. By Theorem \ref{Thm.DPR.1.A}, the distance of two simplices in $\Gamma(d)$ is at most $6d+2$ and the result follows.
\end{proof}


\section{The insertion move for lattice simplices}\label{Sec.DPR.2}

Connecting two polytopes within $\Lambda(d)$ turns out to be much more complicated than within $\Gamma(d)$. Recall that Lemma \ref{Lem.DPR.1.A} makes it obvious that an insertion move is always possible on a polytope when the vertices of this polytope are not constrained to belong to a lattice. Indeed, as already mentioned, one can always insert a point in the polytope, provided this point is close enough to the boundary of the polytope but far enough from its vertices. In the case of lattice polytopes, inserting a point in a neighborhood of the polytope is not always possible. 
Our strategy here is to establish, first, the connectedness of the subgraph induced in $\Lambda(d,k)$ by the simplices and the polytopes with $d+2$ vertices. This is similar to what we did in Section \ref{Sec.DPR.1} with the subgraphs induced in $\Gamma(d)$ by the polytopes with $n$ or $n+1$ vertices, except that here, $n$ has to be equal to $d+1$. 
In this section, we give results on the possibility of inserting a lattice point in a lattice simplex. In particular, we show that, for any positive $k$, there is at least one lattice point in the hypercube $[0,k]^d$ that can be inserted in a given $d$-dimensional lattice simplex contained in $[0,k]^d$.

In the remainder of the section, $S$ denotes a $d$-dimensional lattice simplex contained in the hypercube $[0,k]^d$. For any $i\in\{1, ..., d\}$, we call
$$
\gamma_i^-=\min\{x_i:x\in{S}\}\mbox{ and }\gamma_i^+=\max\{x_i:x\in{S}\}\mbox{.}
$$

Note that, for all $i\in\{1, ..., d\}$, $\gamma_i^-<\gamma_i^+$ because $S$ is $d$-dimensional. The following polytope is the smallest $d$-dimensional combinatorial hypercube containing $S$, and whose facets are parallel to the facets of $[0,k]^d$:
$$
Q=\prod_{i=1}^d[\gamma_i^-,\gamma_i^+]\mbox{.}
$$

Let $R$ be a facet of $Q$. The intersection of $R$ and $S$ is a non-empty, proper face $F$ of $S$. Since $S$ is a simplex, it admits another non-empty face $F^\star$ whose vertices are exactly the vertices of $S$ that do not belong to $F$.

By construction,
$$
\mathrm{dim}(F)+\mathrm{dim}(F^\star)=d-1\mbox{.}
$$

In particular, there exists a vector $c$ that is orthogonal to both $F$ and $F^\star$. Consider the hyperplane $Y$ of $\mathbb{R}^d$ that admits $c$ as a normal vector and such that $F^\star\subset{Y}$. The intersection $S\cap{Y}$ is precisely $F^\star$. Denote by $Y^-$ the closed half-space of $\mathbb{R}^d$ bounded by $Y$ that does not contain $F$.

Since all the vertices of $S$ are incident to either $F$ or $F^\star$, it is an immediate consequence of Lemmas \ref{Lem.DPR.1.A} and \ref{Lem.DPR.1.B} that an insertion move is possible on $S$ for any lattice point in $[0,k]^d$ that does not belong to $S$, to $\mathrm{aff}(F)$, or to $Y^-$. We are now going to search for such lattice points.

Assume, without loss of generality that $c$ is a unit vector and that it points towards $Y^-$. Recall that $R$ is a facet of $Q$ and observe that $\mathrm{aff}(R)\cap[0,k]^d$ is a $(d-1)$-dimensional cube. Denote
\begin{equation}\label{eq.A}
\delta=\min\{c\mathord{\cdot}x:x\in\mathrm{aff}(R)\cap[0,k]^d\}\mbox{.}
\end{equation}

The set
$$
G=\{x\in\mathrm{aff}(R)\cap[0,k]^d:c\mathord{\cdot}x=\delta\}
$$
is a face of $\mathrm{aff}(R)\cap[0,k]^d$. It follows that $G$ is a cube of dimension at most $d-1$. Recall that $c$ is orthogonal to both $F$ and $F^\star$. As a consequence, the map $x\mapsto{c\mathord{\cdot}x}$ is constant within $F$ and within $F^\star$. Call $\varepsilon$ the value of $c\mathord{\cdot}x$ when $x\in{F}$ and $\varepsilon^\star$ the value of $c\mathord{\cdot}x$ when $x\in{F^\star}$. Since $F$ and $Y^-$ are disjoint, $\varepsilon<\varepsilon^\star$. Moreover, by (\ref{eq.A}), $\delta\leq\varepsilon$. Observe that the latter inequality is strict if and only if $F$ is not a subset of $G$. In this case, $F$, $G$, and $Y$ belong to distinct parallel hyperplanes and we immediately obtain the following.
\begin{lem}\label{Lem.BC}
If $F\not\subset{G}$ then $G$ is disjoint from both $\mathrm{aff}(F)$ and $Y^-$.
\end{lem}

In other words, any lattice point in $G$ can be insterted in $S$ in this case. If, on the contrary, $\delta$ and $\varepsilon$ coincide, then $F\subset{G}$. This situation is familiar: we are looking at a lattice simplex $F$ contained in a (possibly degenerate) lattice hypercube $G$. If the dimension of $G$ is greater than the dimension of $F$, then the following lemma provides the desired result.

\begin{lem}\label{Lem.C}
If $k$ is positive and if $P$ is a lattice polytope of dimension less than $d$ contained in $[0,k]^d$, then there exists a lattice point $x$ that belongs to $[0,k]^d$ but that does not belong to the affine hull of $P$.
\end{lem}
\begin{proof}
If $P$ is a lattice polytope of dimension less than $d$ contained in $[0,k]^d$, then the intersection $I$ of its affine hull with $[0,k]^d$ cannot contain more than $(k+1)^{d-1}$ lattice points. Indeed, one can always project $I$ orthogonally on a facet of $[0,k]^d$ in such a way that the dimension of the projection is exactly that of $I$. Such a projection induces an injection from the lattice points in $I$ into the lattice points in the facet of $[0,k]^d$ on which the projection is made.

Now observe that $[0,k]^d$ contains $(k+1)^d$ lattice points. Since $k$ is positive, $(k+1)^{d-1}<(k+1)^d$ and the lemma is proven.
\end{proof}

We now have to address the case when, regardless of which facet $R$ of $Q$ is chosen, $F$ is a subset of $G$ and these polytopes have the same dimension. This case is dealt with by the following lemma.

\begin{lem}\label{Lem.EA}
Call $g$ the maximal dimension of $F$ over all the possible choices for $R$ among the facets of $Q$. Assume that, for any choice of $R$ among the facets of $Q$ such that $F$ has dimension $g$, $F$ is a subset of $G$ and the dimensions of $F$ and $G$ coincide. If $g$ is not greater than $d-2$, then for some choice of $R$ such that $F$ has dimension $g$, there exists a lattice point in $R\mathord{\setminus}[\mathrm{aff}(F)\cup{Y^-}]$.
\end{lem}
\begin{proof}
Consider a facet $R$ of $Q$ such that $F$ has dimension exactly $g$. As $G$ admits $F$ as a subset, $G$ must be a face of $Q$. Taking advantage of the symmetries of $[0,k]^d$, we assume that any facet of $Q$ that contains $G$ is of the form
$$
\{x\in{Q}:x_i=\gamma_i^-\}
$$
for some $i\in\{1, ..., d\}$. In this case, all the coordinates of the vector $c$ are non-negative, except maybe for the coordinate $c_i$ such that all the points $x$ in $R$ satisfy $x_i=\gamma_i^-$. By the maximality of $g$, the intersection of $S$ with any facet of $Q$ incident to $G$ is precisely $F$. In other words, $G$, $F$, $F^\star$, $Y$, and $c$ do not depend on which facet $R$ of $Q$ is chosen, provided this facet is incident to $G$. As a consequence, all the coordinates of $c$ are non-negative.

We will assume that $c_1$ is positive and that
$$
R=\{x\in{Q}:x_1=\gamma_1^-\}\mbox{.}
$$

This can be done without loss of generality by, if needed, permuting the coordinates of $\mathbb{R}^d$. Recall that $F^\star$ is non-empty and consider a vertex $v$ of $F^\star$. By the definition of $\varepsilon^\star$, we have the equality
$$
\sum_{i=1}^dc_iv_i=\varepsilon^\star\mbox{.}
$$

This equality can be transformed into
$$
c_1\gamma_1^-+\sum_{i=2}^dc_iv_i=\varepsilon^\star-c_1(v_1-\gamma_1^-)\mbox{.}
$$

In other words, the orthogonal projection $w$ of $v$ on $R$ (whose coordinates coincide with the coordinates of $v$, except for the first coordinate that is equal to $\gamma_1^-$ instead of $v_1$) satisfies $c\mathord{\cdot}w=\varepsilon^\star-c_1(v_1-\gamma_1^-)$. As $c_1$ is non-zero and as $v_1>\gamma_1^-$, we obtain $c\mathord{\cdot}w<\varepsilon^\star$. It immediately follows that $w\not\in{Y^-}$. Now assume that $g$ is at most $d-2$. In this case, $G$ cannot be a facet of $Q$ and it is incident to at least one facet of $Q$ distinct from $R$. Since $v$ does not belong to any of the facets of $Q$ that contain $G$, its orthogonal projection $w$ on $\mathrm{aff}(R)$ cannot belong to $G$. As a consequence $w$ does not belong to the affine hull of $F$. By construction, $w$ is a lattice point, and the lemma is proven.
\end{proof}

We now state a theorem, obtained by combining Lemmas \ref{Lem.BC}, \ref{Lem.C}, and \ref{Lem.EA}, that will be used in the next section to prove the connectedness of $\Lambda(d,k)$.

\begin{thm}\label{thm.add}
Call $g$ the maximal dimension of $F$ over all the possible choices for $R$ among the facets of $Q$. If $g$ is not greater than $d-2$, then one can choose $R$ among the facets of $Q$ in such a way that $F$ has dimension $g$ and there exists a lattice point in $R\mathord{\setminus}\mathrm{aff}(F)$ that can be inserted in $S$.
\end{thm}

\begin{proof}
Assume that $g\leq{d-2}$. If one can choose $R$ among the facets of $Q$ in such a way that $F$ is $g$-dimensional and $F\not\subset{G}$, then we consider any such facet for $R$ and pick, for $x$, any lattice point in $R$. By Lemma \ref{Lem.BC}, $x$ cannot belong to $\mathrm{aff}(F)$ or to $Y^-$ and, by Lemmas~\ref{Lem.DPR.1.A} and \ref{Lem.DPR.1.B}, it can be inserted in $S$.

Now assume that for any choice of $R$ among the facets of $Q$ such that $F$ has dimension $g$, $F\subset{G}$ but that for some such choice of $R$, the dimension of $F$ is less than the dimension of $G$. In this case, by Lemma \ref{Lem.C}, there exists a lattice point $x$ in $G$ that does not belong to $\mathrm{aff}(F)$. As in addition, $Y^-$ is disjoint from $G$, it follows from Lemmas~\ref{Lem.DPR.1.A} and \ref{Lem.DPR.1.B} that $x$ can be inserted in $S$.

Finally, assume that for any choice of $R$ among the facets of $Q$ such that $F$ has dimension $g$, $F$ is a subset of $G$ and the dimensions of $F$ and $G$ coincide. By Lemma \ref{Lem.EA}, one can choose $R$ such that $F$ has dimension $g$ and there exists a lattice point in $R$ that does not belong to $\mathrm{aff}(F)$ or to $Y^-$. In this case, by Lemmas \ref{Lem.DPR.1.A} and \ref{Lem.DPR.1.B}, $x$ can be inserted in $S$.
\end{proof}

The following corollary shows that there is at least one lattice point in $[0,k]^d$ that can be inserted in $S$. The argument in this proof will be used again in the next section, in order to prove that $\Lambda(d,k)$ is always connected.

\begin{cor}\label{cor.add}
For any positive $k$, an insertion move is possible on $S$ for at least one lattice point contained in the hypercube $[0,k]^d$.
\end{cor}
\begin{proof}
If, for any possible choice of $R$ among the facets of $Q$, the dimension of $F$ is at most $d-2$, then the result follows from Theorem \ref{thm.add}. Assume that, for some facet $R$ of $Q$, $F$ has dimension $d-1$. In this case, $Y$ is parallel to $R$ and $F^\star$ is made up of a single vertex, say $v$. By Lemma \ref{Lem.DPR.1.B}, the intersection of $C_v$ with $Y$ is precisely $v$ and, for every vertex $u$ of $F$, $C_u(S)$ is disjoint from $Y$. Hence, by Lemma \ref{Lem.DPR.1.A}, any lattice point distinct from $v$ in $Y\cap[0,k]^d$ can be inserted in $S$. As $k\geq1$, there exists at least one such lattice point.
\end{proof}

\section{The connectedness of $\Lambda(d)$ and $\Lambda(d,k)$}\label{Sec.DPR.3}

We first prove in this section that $\Lambda(2,k)$ is a connected graph. This will serve as the base case for the inductive proof that $\Lambda(d,k)$ is connected. In the whole section, we call \emph{corner simplex} of $[0,k]^d$ the simplex whose vertices are the origin (the lattice point whose all coordinates are zero), and the $d$ lattice points in $[0,k]^d$ distant from the origin by exactly $1$.

\begin{lem}\label{thm.Connec2}
For any positive $k$, the subgraph induced in $\Lambda(2,k)$ by the triangles and the quadrilaterals is connected.
\end{lem}
\begin{proof}
Since each vertex of a quadrilateral can be deleted, we only need to show that any two triangles are in the same connected component of the subgraph of $\Lambda(2,k)$ induced by triangles and quadrilaterals. Consider a lattice triangle contained in the square $[0,k]^2$. If this triangle does not have a horizontal or a vertical edge then, by Theorem \ref{thm.add}, an insertion move can be performed to transform it into a quadrilateral with a horizontal or a vertical edge, say $e$. It is then possible to delete one of the vertices of  this quadrilateral that is not incident to $e$ in order to obtain a triangle $T$ that admits $e$ as an edge. The strategy is then to transform $T$ into the corner triangle of $[0,k]^2$ using the sequence of moves sketched in Fig. \ref{fig:connect}. This figure shows the case when $e$ is the horizontal edge on the bottom of $T$. In each portion of the figure, the next point for which a move will be performed is colored green or red depending on whether the move is an insertion or a deletion. First observe that a lattice point in the line parallel to $e$ that contains the vertex of $T$ opposite $e$ can be inserted in order to obtain a quadrilateral with three horizontal or vertical edges as shown in the first two portions of Fig. \ref{fig:connect}.
\begin{figure}
\begin{center}
\includegraphics{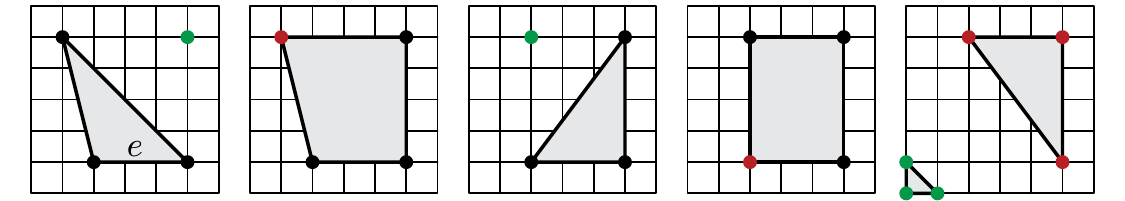}
\caption{An illustration of the sequence of deletion and insertion moves from the lattice triangle shown on the left to the corner triangle with green vertices, on the right.}
\label{fig:connect}
\end{center}
\end{figure}
A deletion move for one of the vertices of the quadrilateral then results in a triangle with a horizontal and a vertical edge as shown in the center of Fig. \ref{fig:connect}. After that, the triangle has a unique oblique edge that faces one of the four vertices of the square $[0,k]^2$. It is always possible to make this edge face the vertex on the bottom left of the square by performing an insertion move to obtain a rectangle and then deleting the bottom-left vertex of the rectangle. This sequence of moves is illustrated in the third and fourth portions of Fig. \ref{fig:connect} in the case when the oblique edge initially faces the top-left vertex of $[0,k]^2$. Finally, one can transform the resulting triangle $U$ into the corner triangle of $[0,k]^2$ (whose vertices are colored green on the right of Fig.~\ref{fig:connect}) by inserting the vertices of the corner simplex one by one, and by deleting a vertex of $U$ after each insertion. Here, one just needs to take care to insert the origin of $\mathbb{R}^2$ first, and to delete the top-right vertex of $U$ last, in the case when $U$ has one or two of its vertices with a zero coordinate.
\end{proof}

We are now ready to prove Theorem \ref{thm.main.2} that can be thought of as the main result of the article. According to it, the subgraph induced in $\Lambda(d,k)$ by simplices and polytopes with $d+2$ vertices is connected.

\begin{proof}[Proof of Theorem \ref{thm.main.2}]
The proof proceeds by induction on $d$. The base case is provided by Lemma \ref{thm.Connec2}. According to Lemma \ref{Lem.DPR.1.BB}, one can always transform a $d$-dimensional polytope with $d+2$ vertices into a lattice simplex by a deletion move. Therefore, we only need to prove that two simplices always are in the same connected component of the subgraph induced in $\Lambda(d,k)$ by simplices and polytopes with $d+2$ vertices. The strategy will be, again, to transform any simplex in this graph into the corner simplex of $[0,k]^d$. Assume that $d\geq3$. Consider a lattice simplex $S$ contained in $[0,k]^d$, and call $\mathcal{V}$ the vertex set of $S$.

As in the previous section, for any $i\in\{1, ..., d\}$, we call
$$
\gamma_i^-=\min\{x_i:x\in{S}\}\mbox{ and }\gamma_i^+=\max\{x_i:x\in{S}\}\mbox{,}
$$
and we consider the combinatorial cube
$$
Q=\prod_{i=1}^d[\gamma_i^-,\gamma_i^+]\mbox{.}
$$

Call $g$ the maximal dimension of the intersection of $S$ and a facet of $Q$. If $g$ is at most $d-2$ then, by Theorem \ref{thm.add}, there exists a facet $R$ of $Q$ such that $S\cap{R}$ is $g$-dimensional and a lattice point $x\in{R\mathord{\setminus}\mathrm{aff}(S\cap{R})}$ that can be inserted in $S$. Consider the polytope $P$ obtained by inserting $x$ in $S$. The intersection $P\cap{R}$ is a simplex because $x\not\in\mathrm{aff}(S\cap{R})$. As a consequence, $P\cap{R}$ is a face of at least one $d$-dimensional simplex that can be obtained by deleting a vertex from $P$. The intersection of this simplex with $R$ is equal to $P\cap{R}$ and therefore has dimension $g+1$. Repeating this procedure provides a sequence of insertion and deletion moves that transform $S$ into a lattice simplex whose intersection $F$ with a facet $R$ of $Q$ has dimension $d-1$

Call $v$ the unique vertex of the lattice simplex that does not belong to $R$. Observe that, in this case, any sequence of insertion and deletion moves that can be performed on $F$ within the cube $\mathrm{aff}(R)\cap[0,k]^d$ can also be performed within $[0,k]^d$ for the pyramid with apex $v$ over $F$. By induction, one can transform $F$ into any lattice simplex contained in the intersection $\mathrm{aff}(R)\cap[0,k]^d$ by carrying out an alternating sequence of insertion and deletion moves in this intersection. This sequence of moves can therefore be performed in order to transform $S$ into the $d$-dimensional lattice simplex $S'$ whose vertex set is made up of $v$, of the lattice point $w$ in $\mathrm{aff}(R)\cap[0,k]^d$ with a unique non-zero coordinate, and of the $d-1$ lattice points in $\mathrm{aff}(R)\cap[0,k]^d$ distant by exactly $1$ from $w$.

Now observe that one can perform an insertion move on any lattice point distinct from $v$ in the intersection of $[0,k]^d$ with the hyperplane parallel to $R$ that contains $v$. We proceed by inserting the lattice point in this intersection whose orthogonal projection on $R$ is $w$ and then, by deleting $v$. Calling $v'$ any of the $d-1$ lattice points in $\mathrm{aff}(R)\cap[0,k]^d$ distant from $w$ by exactly $1$, the simplex that results from the latter deletion is a pyramid with apex $v'$ over a $(d-1)$-dimensional simplex $F'$ such that, for some $i\in\{1, ..., d\}$, $v'$ satisfies $v'_i=1$ and every vertex $u$ of $F'$ satisfies $u_i=0$. Call $R'$ the facet of $[0,k]^d$ made up of the points $x$ such that $x_i=0$. By induction, one can transform $F'$ within $R'$ into the corner simplex of $R'$. From there, one can perform an insertion move on any lattice vertex distinct from $v'$ in the intersection of $[0,k]^d$ with the hyperplane parallel to $R'$ that contains $v'$. We insert the lattice point in this intersection whose orthogonal projection on $R'$ is the origin. Since $v'_i=1$, the $i$-th coordinate of the inserted point is $1$, and its other coordinates are all equal to $0$. Hence, after a last deletion move on $v'$, the resulting simplex is the corner simplex of $[0,k]^d$, which completes the proof. 
\end{proof}

Combining Lemma \ref{Lem.DPR.1.BB} and Theorem \ref{thm.main.2}, we get the following.

\begin{cor}\label{cor.Connecd}
For any positive $k$, $\Lambda(d,k)$ is connected.
\end{cor}

We now turn our attention to the connectedness of $\Lambda(d)$.

\begin{thm}\label{thm.Connect.lattice}
For any $d\geq2$, both $\Lambda(d)$ and the subgraph induced in $\Lambda(d)$ by simplices and polytopes with $d+2$ vertices are connected.
\end{thm}
\begin{proof}
Note that the translations of $\mathbb{Z}^d$ by lattice vectors induce automorphisms of $\Lambda(d)$. Since two $d$-dimensional lattice polytopes contained in $\mathbb{R}^d$ can always be displaced into the hypercube $[0,k]^d$ for some large enough $k$ by a such a translation, they both belong to a subgraph of $\Lambda(d)$ isomorphic to $\Lambda(d,k)$. The connectedness of $\Lambda(d)$ therefore follows from Corollary \ref{cor.Connecd}. For the same reason, the connectedness of the subgraph induced in $\Lambda(d)$ by simplices and polytopes with $d+2$ vertices follows from Theorem \ref{thm.main.2}.
\end{proof}

\section{The number of possible insertion and deletion moves}\label{Sec.DPR.4}

The main purpose of this section is to study how the vertex degrees in $\Lambda(d)$ and $\Lambda(d,k)$ decompose between insertion and deletion moves. In particular, we will exhibit a family of polytopes whose dimension and number of vertices can be arbitrarily large, but in which no lattice point can be inserted. Hence, the vertex degrees in $\Lambda(d)$ can be finite. 
We then turn our attention to $\Lambda(d,k)$. The vertex degrees in this graph are bounded above by $(k+1)^d$, the number of lattice points in $[0,k]^d$. This bound is obviously sharp when $k=1$ since all the lattice points in $[0,1]^d$ can be deleted from the hypercube itself. We will show that this bound is also sharp when both $d$ and $k$ grow large by exhibiting an extensive family of $d$-dimensional lattice polytopes contained in $[0,k]^d$ such that every lattice point in $[0,k]^d$ can either be inserted in or deleted from these polytopes. We first prove the following about $\Lambda(2)$. Thereafter, by a unit square we mean the square $[0,1]^2$ or any of its translates by a lattice vector.

\begin{lem}\label{Lem.DPR.4.1}
For any $n>3$ distinct from $5$, there exists a lattice polygon $P\subset\mathbb{R}^2$ with $n$ vertices such that no point of $\mathbb{Z}^2$ can be inserted in $P$. 
\end{lem}
\begin{proof}
First observe that if $P$ is a unit square, then every point in the lattice $\mathbb{Z}^2$ is contained in the cone $C_v(P)$, where $v$ is one of the four vertices of $P$. It then follows from Lemma \ref{Lem.DPR.1.A} that no point of $\mathbb{Z}^2$ can be inserted in $P$, which proves the lemma when $n=4$.

Now assume that $n\geq6$ and consider the map
$$
f:x\mapsto{x(x-1)/2}\mbox{.}
$$

Let $\mathcal{A}$ be the set of the points $x$ in $\mathbb{Z}^2$ such that $x_2=f(x_1)$. Note that these points are the vertices of a convex polygonal line. We are going to build a polygon $P_n$ from this polygonal line.

First assume that $n$ is even and consider the point $a$ satisfying
$$
a_1=\frac{n}{2}-1\mbox{ and }a_2=f(a_1)+1\mbox{.}
$$

Let $P_n$ denote the polygon whose vertices are the elements $x$ in $\mathcal{A}$ such that $0\leq{x_1}<n/2$ and their symmetric with respect to the point $a/2$. This polygon is depicted in Fig. \ref{Fig.DPR.4.1} when $n$ is equal to $6$ (left), $8$ (center), and $10$ (right). By construction, $P_n$ is centrally-symmetric and its centroid is $a/2$. Note that it has $n$ vertices, half of whose belong to $\mathcal{A}$. Further note that $a$ and the point $b$ such that $b_1=a_1-1$ and $b_2=a_2$ are the two vertices of an horizontal edge of $P_n$. In the figure, the portion of $\mathbb{R}^2$ covered by the cones $C_v(P_n)$, where $v$ ranges over the vertices of $P_n$, is colored red.
\begin{figure}
\begin{center}
\includegraphics{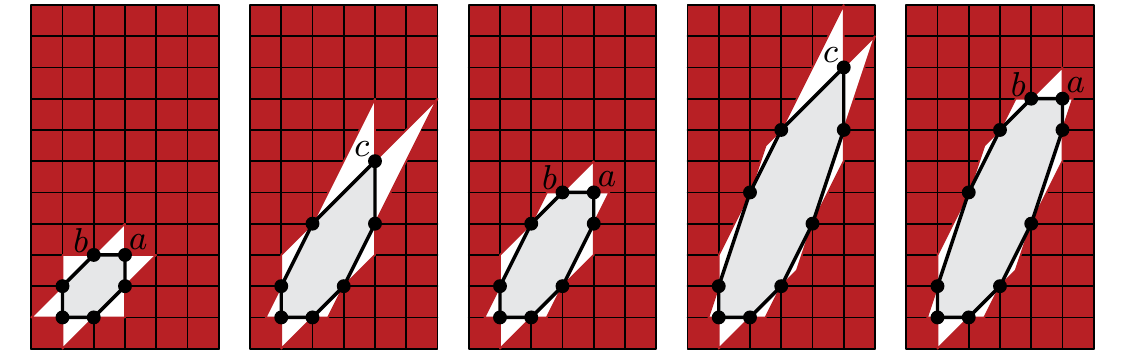}
\caption{The polygon $P_n$ when $n$ ranges from $6$ to $10$.}
\label{Fig.DPR.4.1}
\end{center}
\end{figure}
Observe that the portion of $\mathbb{R}^2$ that is not covered by $P_n$ or by any of the cones $C_v(P_n)$, where $v$ is a vertex of $P_n$, is the union of the interiors of a set of triangles colored white in the figure. By construction, each of these triangles is contained in the region of $\mathbb{R}^d$ made up of the points $x$ such that $i\leq{x_1}\leq{i+1}$, for some integer $i$. Hence, the interiors of these triangles entirely avoid the lattice $\mathbb{Z}^2$. It then follows from Lemma \ref{Lem.DPR.1.A} that no point in $\mathbb{Z}^2$ can be inserted in $P_n$.

Now assume that $n$ is odd. Let $P_n$ be the polygon obtained as the convex hull of $P_{n+1}$ and of the point $c$ such that $c_1=a_1$ and $c_2=a_2+1$. The polygon $P_n$ is depicted in Fig. \ref{Fig.DPR.4.1} when $n$ is equal to $7$ (second polygon from the left) and $9$ (next-to-last polygon). While $c$ is a vertex of $P_n$, $a$ and $b$ are no longer vertices of it because they are contained in the relative interiors of the edges of $P_n$ incident to $c$. As $P_{n+1}$ shares all its vertices with $P_n$ except for $a$ and $b$, $P_n$ has exactly $n$ vertices (one less than $P_{n+1}$). As above, the portion of $\mathbb{R}^2$ that is not covered by $P_n$ or by any of the cones $C_v(P_n)$, where $v$ is a vertex of $P_n$, is the union of the interiors of a set of triangles. As we have seen above, the interiors of all these triangles are disjoint from $\mathbb{Z}^2$, except possibly for the two triangles incident to $c$, that have been introduced when building $P_n$ from $P_{n+1}$. Among these two triangles, the one depicted on top of the polygon in Fig. \ref{Fig.DPR.4.1} does not depend on $n$, and it can be seen in the figure that its interior is disjoint from $\mathbb{Z}^2$. The other triangle depends on $n$. As can be seen in the figure, its interior is disjoint from $\mathbb{Z}^2$ when $n=7$. When $n\geq9$, this triangle is contained in the region of $\mathbb{R}^d$ made up of the points $x$ such that $n/2-1\leq{x_1}\leq{n/2}$, and its interior is therefore also disjoint from $\mathbb{Z}^2$. Hence, by Lemma \ref{Lem.DPR.1.A}, no point of the lattice can be inserted in $P_n$.
\end{proof}

Theorem \ref{thm.main.3} is an immediate consequence of Lemma \ref{Lem.DPR.4.1}: the subgraph induced in $\Lambda(2)$ by the polygons with $n$ or $n+1$ vertices is disconnected when $n$ is distinct from $3$ and $5$. According to Theorem \ref{thm.Connect.lattice}, this subgraph is connected when $n=3$. The exception for $n=5$ may look odd at first but turns out to make sense. Indeed, we shall see in the next section that the subgraph induced in $\Lambda(2)$ by pentagons and hexagons is connected.

Note that Lemma \ref{Lem.DPR.4.1} only exhibits isolated vertices within the subgraph induced in $\Lambda(2)$ by the polygons with $n$ or $n+1$ vertices. It turns out that this subgraph can have an infinite number of connected components larger than just an isolated vertex. Indeed, consider the polygon $P_{10}$ shown on the right of Fig.~\ref{Fig.DPR.4.1}, delete a vertex from this polygon, and call the resulting polygon $Q$. One can see that the deleted vertex is the only lattice point that can be inserted in $Q$. Since this holds for any vertex of $P_{10}$, we get a connected component of the subgraph induced in $\Lambda(2)$ by enneagons and decagons that contains exactly one decagon and ten enneagons. There is one such connected component for any translation of $P_{10}$ by a lattice vector. Thus, the subgraph induced in $\Lambda(2)$ by enneagons and decagons admits infinitely many connected components with $11$ vertices. We now generalize Lemma \ref{Lem.DPR.4.1} by showing that there are lattice polytopes of arbitrarily large dimension whose number of vertices is also arbitrarily large such that no lattice point can be inserted.

We will first need the following lemma.

\begin{lem}\label{Lem.DPR.4.4}
Let $P$ and $Q$ be two polytopes. If $u$ and $v$ are a vertex of $P$ and a vertex of $Q$, respectively, then $C_{u\mathord{\times}v}(P\mathord{\times}Q)$ coincides with $C_u(P)\mathord{\times}C_v(Q)$.
\end{lem}
\begin{proof}
Consider two polytopes $P$ and $Q$ which, we assume are a $p$-dimensional and a $q$-dimensional polytope contained in $\mathbb{R}^p$ and $\mathbb{R}^q$, respectively. Let $u$ be a vertex of $P$ and $v$ a vertex of $Q$. The facets of $P\mathord{\times}Q$ incident to $u\mathord{\times}v$ are precisely the cartesian products of the form $F\mathord{\times}Q$ where $F$ is a facet of $P$ incident to $u$ and $P\mathord{\times}G$ where $G$ is a facet of $Q$ incident to $v$.

In particular, if $F$ is a facet of $P$ incident to $u$, then
$$
H^-_{F\mathord{\times}Q}(P\mathord{\times}Q)=H^-_F(P)\mathord{\times}\mathbb{R}^q\mbox{.}
$$

Similarly, if $G$ is a facet of $Q$ incident to $v$, then
$$
H^-_{P\mathord{\times}F}(P\mathord{\times}Q)=\mathbb{R}^p\mathord{\times}H^-_G(Q)\mbox{.}
$$

As a consequence, for any point $x$ in $C_u(P)$ and any point $y$ in $C_v(Q)$, $x\mathord{\times}y$ is contained both in $H^-_{F\mathord{\times}Q}(P\mathord{\times}Q)$ and in $H^-_{P\mathord{\times}G}(P\mathord{\times}Q)$. Inversely, if $x$ and $y$ are two points in $\mathbb{R}^p$ and $\mathbb{R}^q$, respectively, such that $x\mathord{\times}y$ is contained in $C_{u\mathord{\times}v}(P\mathord{\times}Q)$, then $x$ necessarily belongs to $H^-_{F}(P)$ and $y$ to $H^-_{F}(Q)$. Since these two statements hold for any facet $F$ of $P$ incident to $u$ and any facet $G$ incident to $v$, we obtain the desired equality.
\end{proof}

We prove the following result by considering cartesian products of polygons and hypercubes, for which no insertion of a lattice point is possible.

\begin{thm}\label{Thm.DPR.4.4}
For all $n>3$ such that $n\neq5$, and for all $d\geq{4}$, there exists a $d$-dimensional lattice polytope $P$ contained in $\mathbb{R}^d$ with $n2^{d-2}$ vertices such that no point in the lattice $\mathbb{Z}^d$ can be inserted in $P$. 
\end{thm}
\begin{proof}
Let $n$ be an integer greater than $3$ and distinct from $5$. By Lemma \ref{Lem.DPR.4.1}, there exists a lattice polygon $Q$ with $n$ vertices such that no point of $\mathbb{Z}^2$ can be inserted in $Q$. Now assume that $d\geq4$, and recall that no point in $\mathbb{Z}^{d-2}$ can be inserted in the hypercube $[0,1]^{d-2}$. By Lemmas \ref{Lem.DPR.1.A} and \ref{Lem.DPR.4.4}, no point in $\mathbb{Z}^d$ can be inserted in $Q\mathord{\times}[0,1]^{d-2}$. This cartesian product is a $d$-dimensional lattice polytope with $n2^{d-2}$ vertices, as desired. 
\end{proof}

We now turn our attention to another interesting family of lattice polytopes. These polytopes play a peculiar role in the graph $\Lambda(d,k)$: they are the polytopes $P$ such that all the lattice points in $[0,k]^d$ can either be inserted in $P$ or deleted from it. When $k=1$, these polytopes admit a straightforward characterization: they are the $d$-dimensional polytopes contained in $[0,1]^d$ that are not pyramids over any of their facets.
\begin{figure}[b]
\begin{center}
\includegraphics{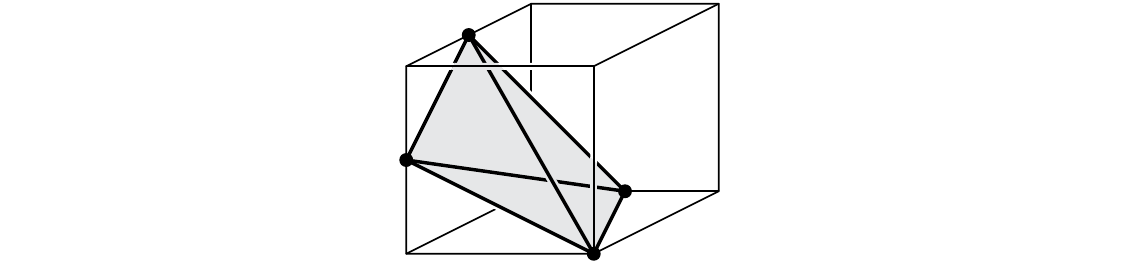}
\caption{An empty tetrahedron.}
\label{Fig.DPR.4.45}
\end{center}
\end{figure}
For this reason, we will assume that $k\geq2$ in the remainder of the section. The polytopes we are looking for are necessarily \emph{empty} lattice polytopes in the sense that their intersection with $\mathbb{Z}^d$ is precisely their vertex set. We will build them as cartesian products of empty simplices due to B{\'a}r{\'a}ny and Seb\H{o} (see \cite{HaaseZiegler2000,Sebo1999}), such as the empty tetrahedron depicted in Fig.~\ref{Fig.DPR.4.45} inside the cube $[0,2]^3$. Note that the property we are investigating here cannot carry over to the whole lattice $\mathbb{Z}^d$. Indeed, given a lattice polytope $P$ there is always a point in $\mathbb{Z}^d$ that cannot be inserted in or deleted from $P$: any lattice point in the affine hull of an edge of $P$ will have this property as soon as it is distinct from the two extremities of this edge.

\begin{thm}\label{Thm.DPR.4.5}
Consider an integer $k\geq2$. If $k+1$ is a proper divisor of $d$, then there exists a $d$-dimensional lattice polytope $P$ with $(k+2)^{d/(k+1)}$ vertices contained in $[0,k]^d$ such that, for any lattice point $x$ in $[0,k]^d$, $x$ can either be inserted in $P$ or removed from it. 
\end{thm}
\begin{proof}
Consider the following matrix with $k+2$ columns and $k+1$ rows:
$$
\left[
\begin{array}{ccccccc}
k & 1 & 0 & \cdots & 0 & 0 & 0 \\
0 & k & 1 & \ddots & \vdots & \vdots & \vdots\\
0 & 0 & k& \ddots & 0 & 0 & 0 \\
0 & 0 & 0 & \ddots & 1 & 0 & 0 \\
\vdots & \vdots& \vdots & \ddots & k & 1 & 0\\
0 & 0 & 0 & \cdots & 0 & k & 1\\
\end{array}
\right]
$$

Denote by $a(i)$ the point of $\mathbb{R}^{k+1}$ whose vector of coordinates is the $i$-th column of this matrix. These points are the vertices of a $(k+1)$-dimensional empty simplex $S$ \cite{HaaseZiegler2000,Sebo1999}. When $k$ is equal to $2$, $S$ is the empty tetrahedron depicted in Fig.~\ref{Fig.DPR.4.45} inside the cube $[0,2]^3$. Observe that each of the vertices of $S$ is either a vertex of the hypercube $[0,k]^{k+1}$ or contained in an edge of this hypercube: $a(1)$ is a vertex of $[0,k]^{k+1}$ and, when $2\leq{i}\leq{k+2}$, $a(i)$ is contained in the edge of $[0,k]^{k+1}$ whose two vertices are obtained from $a(i)$ by replacing its $(i-1)$-th coordinate by $0$ or $k$. Further observe that each of these edges contains exactly one vertex of $S$. In other words, whenever $1\leq{i}\leq{k+2}$, we can find a face $F$ of $[0,k]^{k+1}$ that contains $a(i)$ and no other vertex of $S$. Consider an hyperplane $H$ of $\mathbb{R}^{k+1}$ whose intersection with $[0,k]^{k+1}$ is $F$. Denote by $H^-$ the closed half-space of $\mathbb{R}^{k+1}$ bounded by $H$ such that $[0,k]^{k+1}\cap{H^-}=F$. Since $a(i)$ is the only vertex of $S$ contained in $F$, then $S\cap{H}=\{a(i)\}$. By Lemma \ref{Lem.DPR.1.B}, the intersection of $[0,k]^{k+1}$ with $C_{a(i)}(S)$ is therefore precisely $\{a(i)\}$. According to Lemma \ref{Lem.DPR.1.A} all the lattice points in $[0,k]^{k+1}$ can be inserted in $S$, except for the vertices of $S$.

Now assume that $k+1$ is a divisor of $d$ and denote by $P$ the cartesian product $S^{d/(k+1)}$. It follows from Lemma \ref{Lem.DPR.4.4} that every lattice point in $[0,k]^d$ can be inserted in $P$ except for the vertices of $P$. Consider a facet $G$ of $P$. This facet is obtained as the cartesian product of $d/(k+1)-1$ copies of $S$ with a facet $F$ of $S$. If $F$ is the $j$-th term of the product, the vertices of $P$ that are not incident to $G$ are precisely the cartesian products of $d/(k+1)-1$ (possibly not pairwise distinct) vertices of $S$ such that the $j$-th term in the product is equal to the vertex of $S$ not incident to $F$. Since $k+1$ is a proper divisor of $d$, then there are several such vertices and $P$ cannot be a pyramid over any of its facets. As a consequence, all the vertices of $P$ can be deleted.
\end{proof}

Theorem \ref{Thm.DPR.4.5} does not hold in dimension $2$.

\begin{thm}\label{Thm.DPR.4.6}
Consider a positive integer $k\geq2$. If $P$ is a lattice polygon contained in $[0,k]^2$, then there exists a lattice point in $[0,k]^2$ that cannot be inserted in $P$ or a vertex of $P$ that cannot be deleted from $P$. 
\end{thm}
\begin{proof}
Consider a lattice polygon $P$ contained in $[0,k]^2$. Denote by $a$ and $b$ two vertices of $P$ whose distance is the largest possible. Note that the distance of $a$ and $b$ is then at least $\sqrt{2}$. In particular, if $a_1=b_1$ or if $a_2=b_2$, then the convex hull of $a$ and $b$ contains at least one lattice point in its interior. This lattice point cannot be inserted in or deleted from $P$, and the theorem holds in this case. In the following we assume that $a_1\neq{b_1}$ and $a_2\neq{b_2}$. Using the symmetries of the lattice, we can assume without loss of generality that $a_i<b_i$ when $i\in\{1,2\}$. Consider the rectangle $[a_1,b_1]\mathord{\times}[a_2,b_2]$. 

First assume that $P$ has a vertex $c$ outside of the rectangle $[a_1,b_1]\mathord{\times}[a_2,b_2]$. Taking advantage of the symmetries of that rectangle, we can assume without loss of generality that $c_1>b_1$. In this case, $c_2$ is necessarily less than $b_2$ because $c$ is at most as distant from $a$ than $b$. If $a_2\leq{c_2}<b_2$, then the lattice point $x$ such that $x_1=b_1$ and $x_2=c_2$ is in the triangle with vertices $a$, $b$, and $c$ and it is distinct from its three vertices. Hence, $x$ cannot be inserted in or deleted from $P$ because it is contained in $P$ and it is distinct from all the vertices of $P$. If $c_2<a_2$, then the lattice point $x$ such that $x_1=b_1$ and $x_2=a_2$ is in the interior of the triangle with vertices $a$, $b$, and $c$. As above, this point cannot be inserted in $P$ or deleted from it, proving the theorem in this case.

Now assume that $P$ has a vertex $c$ distinct from $a$ and $b$ inside the rectangle $[a_1,b_1]\mathord{\times}[a_2,b_2]$. If $c$ is in some edge of that rectangle, then $P$ has a horizontal or a vertical edge. None of the lattice points in $[0,k]^2$ that belong to the affine hull of that edge can be inserted in $P$ and, since $k\geq2$, at least one of these lattice points is not a vertex of $P$. This point therefore cannot be inserted in or deleted from $P$, and the theorem holds in this case. Finally, if $c$ is in $]a_1,b_1[\mathord{\times}]a_2,b_2[$, we can assume without loss of generality that $c$ is below the affine hull of $a$ and $b$. In this case, $C_c(P)$ contains all the points $x$ such that $x_1\geq{c_1}$ and $x_2=c_2$. In particular, the point $x$ such that $x_1=b_1$ and $x_2=c_2$ belongs to the interior of this cone, which completes the proof.
\end{proof}

\section{The subgraph of $\Lambda(2)$ induced by pentagons and hexagons}\label{Sec.DPR.5}

According to Theorem \ref{thm.main.3}, the subgraph induced in $\Lambda(2)$ by the polygons with $n$ or $n+1$ vertices is always  disconnected except possibly when $n$ is equal to $3$ and $5$. By Theorem \ref{thm.Connect.lattice}, this subgraph is connected when $n=3$. In this section, we deal with the remaining case. As a first step, we describe a large connected component of the subgraph induced in $\Lambda(2)$ by the polygons with $n$ or $n+1$ vertices.
We will need the following notion.

\begin{defn}\label{defn.wedge}
A polygon $P$ will be called \emph{oblique} if it admits two consecutive vertices $a$ and $b$ such that for every vertex $v$ of $P$ distinct from $a$ and from $b$, the inequalities $a_1<v_1<b_1$ and $a_2<v_2<b_2$ hold.  
\end{defn}

An example of an oblique polygon $P$ is depicted on the left of Fig. \ref{Fig.DPR.4.2}.

\begin{lem}\label{Lem.DPR.4.2}
Consider an oblique lattice polygon $P$ with $n$ vertices. The translates of $P$ by a lattice vector, and the lattice polytopes centrally symmetric to $P$ with respect to a point of the plane all belong to the same connected component of the subgraph induced in $\Lambda(2)$ by the polygons with $n$ or $n+1$ vertices. 
\end{lem}
\begin{proof}
Let $a$ and $b$ be the two consecutive vertices of $P$ such that for any vertex $v$ of $P$ distinct from $a$ and $b$, the inequalities $a_1<v_1<b_1$ and $a_2<v_2<b_2$ hold. Denote by $E$ the edge of $P$ with vertices $a$ and $b$. We start by proving that the lattice polygon $Q$ symmetric to $P$ with respect to the center of $E$ can be reached from $P$ by inserting a sequence of lattice points and deleting a vertex immediately after each insertion. Since $P$ is oblique, the vertices of the convex hull of $P\cup{Q}$ are exactly the vertices of $P$ and the vertices of $Q$.
\begin{figure}
\begin{center}
\includegraphics{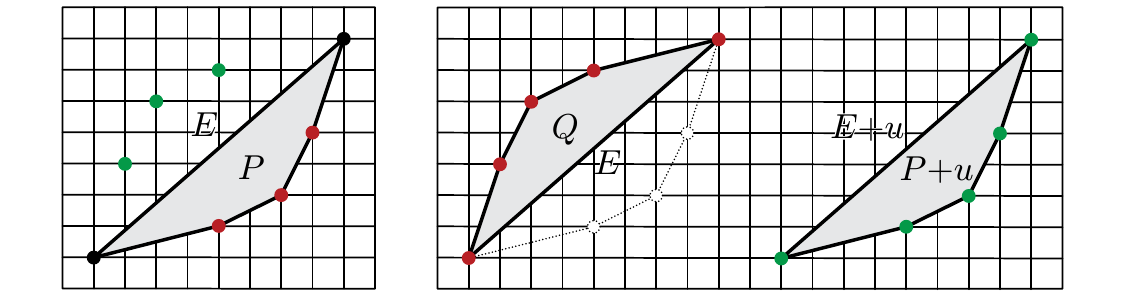}
\caption{Symmetrizing and translating an oblique lattice polygon.}
\label{Fig.DPR.4.2}
\end{center}
\end{figure}
Lemma \ref{Lem.DPR.1.BC}, then provides the desired sequence of moves.

Now, we show that the translate of $P$ by a lattice vector $u$ can be reached using an appropriate sequence of elementary moves. Decomposing this translation along the two directions of the lattice, we can restrict to showing this property when $u_2=0$. We can assume without loss of generality that $u_1>0$ by exploiting the symmetries of $\mathbb{Z}^2$ and we can also assume without loss of generality that $H^-_E(P)$ does not contain $E+u$ by, if needed, considering the reverse transformation from $P+u$ to $P$ instead of the transformation from $P$ to $P+u$. We first perform a sequence of moves that transform $P$ into its symmetric $Q$ with respect to the center of $E$, as explained above. After that, note that $P+u$ is a subset of $H^-_E(Q)$. In fact, the convex hull of $Q$ and $P+u$ admits, as its vertices, the vertices of $Q$ and the vertices of $P+u$, as shown on the right of Fig. \ref{Fig.DPR.4.2}. Lemma \ref{Lem.DPR.1.BC}, then again provides the desired sequence of moves.

Finally, combining the symmetrization operation and the translation operation shows that all the lattice polygons centrally symmetric to $P$ with respect to a point of the plane can be obtained from $P$ by inserting lattice points, and deleting vertices immediately after each insertion.
\end{proof}

We now prove that oblique lattice polygons with the same number of vertices can always be changed into one another in $\Lambda(2)$ by a sequence of elementary moves that alternate between insertions and deletions.

\begin{lem}\label{Lem.DPR.4.3}
For any $n\geq3$, the oblique lattice polygons with $n$ vertices all belong to the same connected component of the subgraph induced in $\Lambda(2)$ by the polygons with $n$ or $n+1$ vertices. 
\end{lem}
\begin{proof}
Assume that $P$ and $Q$ are two oblique lattice polygons. Let $a$ and $b$ be the consecutive vertices of $P$ such that for any vertex $v$ of $P$ distinct from $a$ and $b$, $a_1<v_1<b_1$ and $a_2<v_2<b_2$. Denote by $E$ the edge of $P$ with vertices $a$ and $b$ and by $F$ the corresponding edge of $Q$. Let $L$ be the line made up of the points $x\in\mathbb{R}^2$ such that $x_2=0$. By Lemma \ref{Lem.DPR.4.2}, we can translate $P$ and $Q$ in such a way that $E$ and $F$ both intersect $L$ in their relative interiors. We can also require this translation to send $Q$ into $H^-_E(P)$. We further require that $P$ is a subset of $H^-_F(Q)$ by, if needed, using a sequence of moves that transform $Q$ into its symmetric with respect to the center of $F$. Now observe that, if $P$ and $Q$, after these operations, are sufficiently far apart along $L$, then the convex hull of their union admits, for its vertex set, the union of the vertex sets of $P$ and $Q$. The result then follows from Lemma \ref{Lem.DPR.1.BC}.
\end{proof}

The next step consists in showing that the connected component of the subgraph induced in $\Lambda(2)$ by the polygons with $n$ or $n+1$ vertices that contains the oblique lattice polygons also contains a larger class of polygons.

\begin{defn}\label{defn.flat}
A polygon $P$ will be called \emph{flat} if there exist a non-zero lattice vector $c$ and two consecutive vertices $a$ and $b$ of $P$ such that for every vertex $v$ of $P$ distinct from $a$ and from $b$, $a\mathord{\cdot}c\leq{v\mathord{\cdot}c}\leq{b\mathord{\cdot}c}$. If these inequalities are strict for every vertex $v$ of $P$, then $P$ is called \emph{strongly flat}.
\end{defn}

A flat lattice polygon $P$ is depicted on the left of Fig. \ref{Fig.DPR.4.3}, where the vector $c$ mentioned in Definition \ref{defn.flat} is colored blue. By this definition, the lines orthogonal to $c$ through the vertices of $P$, colored blue in the figure, all intersect one of the edges of $P$, labelled $E$ in the figure. If these lines are pairwise distinct or, equivalently, if no edge of the polygon is orthogonal to $c$, then the polygon is strongly flat. We have the following.

\begin{lem}\label{Lem.DPR.4.4}
Let $P$ be a flat lattice polygon with $n$ vertices. If $n\geq5$, then $P$ is connected to a strongly flat polygon with $n$ vertices by a path in the subgraph induced in $\Lambda(2)$ by the polygons with $n$ or $n+1$ vertices. 
\end{lem}
\begin{proof}
Label the vertices of $P$ clockwise from $p^1$ to $p^n$ in such a way that, for some non-zero lattice vector $c$, the inequalities $p^1\mathord{\cdot}c\leq{p^i\mathord{\cdot}c}\leq{p^n\mathord{\cdot}c}$ hold whenever $1<i<n$.
\begin{figure}
\begin{center}
\includegraphics{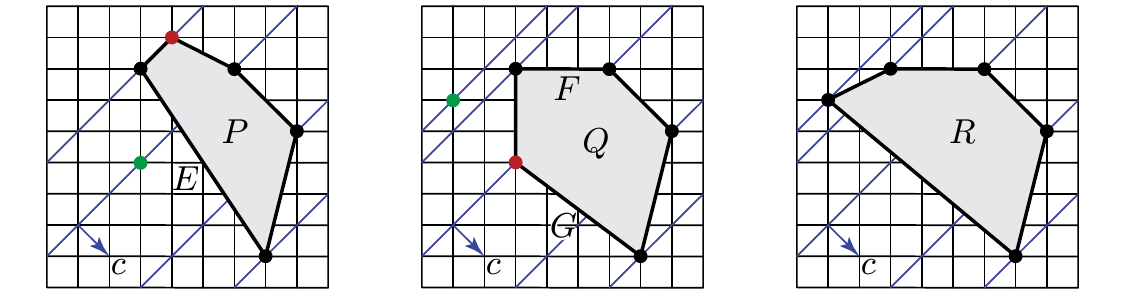}
\caption{Making a flat lattice polygon strongly flat.}
\label{Fig.DPR.4.3}
\end{center}
\end{figure}
By the convexity of $P$, all of these inequalities are strict, except possibly for $p^1\mathord{\cdot}c\leq{p^2\mathord{\cdot}c}$ and $p^{n-1}\mathord{\cdot}c\leq{p^n\mathord{\cdot}c}$. The former inequality turns into an equality precisely when the segment between $p^1$ and $p^2$ is orthogonal to $c$, and the latter when the segment between $p^{n-1}$ and $p^n$ is orthogonal to $c$. Observe that these segments are the only two possible edges of $P$ that can be orthogonal to $c$. Call $E$ the edge with vertices $p^1$ and $p^n$.

Assuming that $n\geq5$, the line orthogonal to $c$ through $p^3$ must intersect $E$ in its relative interior. Since $P$ is flat, all the lattice points in this line that are separated from $p^3$ by the affine hull of $E$ can be inserted in $P$. There is an infinite number of such lattice points because $c$ has integer coordinates. Let $x$ be one of these lattice points. Call $Q$ the polygon obtained by inserting $x$ in $P$ and then deleting $p^2$ from it. These two operations are sketched on the left of Fig. \ref{Fig.DPR.4.3}, where $x$ is colored green and $p^2$ is colored red.

Now call $y=x-kc$, where $k$ is a positive integer. We can choose $k$ large enough so that $y\mathord{\cdot}c<p^1\mathord{\cdot}c$. In the center of Fig. \ref{Fig.DPR.4.3}, the point $y$ is colored green and, as can be seen, $k$ is taken equal to $2$. Observe that the segment $F$ with vertices $p^1$ and $p^3$ is an edge of $Q$ whose affine hull is not orthogonal to $c$. As a consequence, there exists an infinite number of lattice points outside of $H_F^-(Q)$ that belong to the line orthogonal to $c$ through $y$. Now recall that $x$ can be chosen arbitrarily far away from $P$. Hence, one can pick $x$ in such a way that $y$ does not belong to $H_F^-(Q)$. Note that this amounts to translate the segment with vertices $x$ and $y$ away from $P$. Observe that $y$ may still belong to $H_G^-(Q)$, where $G$ is the edge of $Q$ with vertices $x$ and $p^n$. However, in this case, $y$ can still be placed outside of $H_G^-(Q)$ by further translating $x$ away from $P$. In this case, $y$ can be inserted in $Q$. Call $R$ the polygon obtained by inserting $y$ in $Q$ and then deleting $x$ from it. These two operations are sketched in the center of Fig. \ref{Fig.DPR.4.3}, where $y$ is colored green and $x$ is colored red.

Repeating the whole procedure using the edge of $R$ with extremities $p^{n-1}$ and $p^n$ allows to build a strongly flat lattice polygon from $P$. By construction, the path from $P$ to this polygon is contained in the subgraph induced in $\Lambda(2)$ by the polygons with $n$ or $n+1$ vertices.
\end{proof}

Strongly flat lattice polygons are, in turn, connected to oblique lattice polygons via sequences of moves that alternate between the insertion of a lattice point and the deletion of a vertex. More precisely, we have the following.

\begin{lem}\label{Lem.DPR.4.5}
For any integer $n$ such that $n\geq3$, the strongly flat lattice polygons with $n$ vertices all belong to the same connected component of the subgraph induced in $\Lambda(2)$ by the polygons with $n$ or $n+1$ vertices. 
\end{lem}
\begin{proof}
If follows from Definitions \ref{defn.wedge} and \ref{defn.flat} that an oblique polygon is necessarily strongly flat. The result can therefore be obtained from Lemma \ref{Lem.DPR.4.3} by showing that a strongly flat lattice polygon with $n$ vertices can always be transformed into an oblique lattice polygon with $n$ vertices by a sequence of moves such that each insertion move is followed by a deletion move.

Let $P$ be a strongly flat lattice polygon. Index the vertices of $P$ clockwise from $p^1$ to $p^n$ in such a way that, for some non-zero lattice vector $c$, the inequalities $p^1\mathord{\cdot}c<p^i\mathord{\cdot}c<p^n\mathord{\cdot}c$ hold whenever $1<i<n$. Since these inequalities are strict, they all remain true when $c$ is replaced by its sum with a unit lattice vector, provided $c$ is long enough. As $c$ can be taken arbitrarily long, we can therefore assume without loss of generality that both $c_1$ and $c_2$ are non-zero. Now consider a non-zero lattice vector $u$ orthogonal to $c$. Since both coordinates of $c$ are non-zero, then so are the two coordinates of $u$.

Call $E$ the segment with vertices $p^1$ and $p^n$. We require that $u$ points toward $H_E^-$ by, if needed, replacing $u$ by $-u$. Let $q^i$ denote the lattice point centrally-symmetric to $p^i$ with respect to the center of $E$. The points $q^1$ to $q^n$ are the vertices of the lattice polygon $Q$ that is centrally-symmetric to $P$ with respect to the center of $E$. We are now going to shear $Q$ into an oblique lattice polytope $\psi(Q)$, as illustrated in Fig. \ref{Fig.DPR.4.4}.
\begin{figure}[b]
\begin{center}
\includegraphics{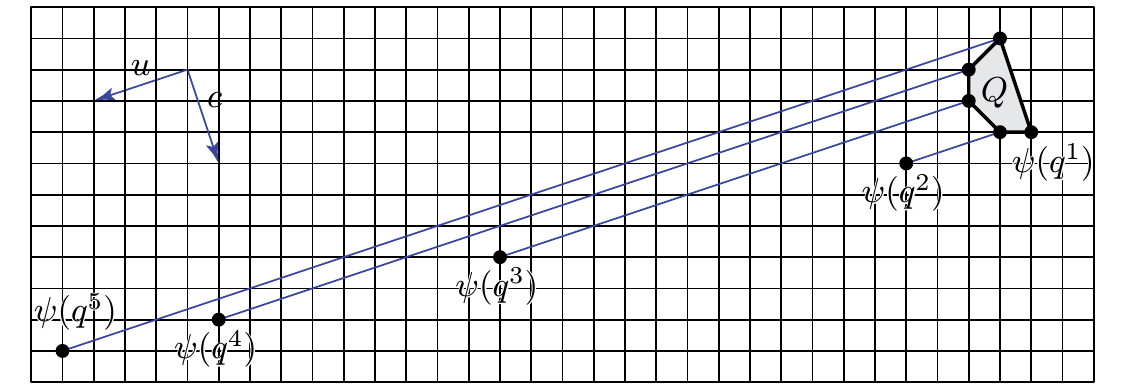}
\caption{Making a strongly flat lattice polygon oblique.}
\label{Fig.DPR.4.4}
\end{center}
\end{figure}
Since $u_1\neq0$, for any point $x\in\mathbb{R}^2$, the difference $x-p^1$ can be written uniquely as the sum of a vertical vector with a vector parallel to $u$. Denote by $\phi(x)$ the second coordinate of the vertical vector. Note that, if $x$ is a lattice point, then $\phi(x)$ is rational but it is not necessarily an integer. For instance, in the case shown in Fig. \ref{Fig.DPR.4.4}, $\phi(q^1)=0$, $\phi(q^2)=1/3$, $\phi(q^3)=5/3$, $\phi(q^4)=8/3$, and $\phi(q^5)=10/3$. Further note that, in general, $\phi(q^1)$ is equal to $0$ and the other $\phi(q^i)$ are either all positive or all negative. Now consider the following affine transformation of $\mathbb{R}^2$:
$$
\psi:x\mapsto{x+k\phi(x)u}\mbox{,}
$$
where $k$ is an integer such that, for all $i\in\{1, ..., n\}$, $k\phi(q^i)\in\mathbb{N}$. Observe that the integer $k$ necessarily exists because $\phi(q^i)$ is rational for all $i\in\{1, ..., n\}$. By our choice for $k$, $\psi(q^i)$ is a lattice point. In fact, there is an infinite number of such values for $k$ since it can be replaced by its product with any positive integer. We can therefore also require that, whenever $1\leq{i}<n$,
\begin{equation}\label{Lem.DPR.4.5.Eq.1}
\left\{
\begin{array}{l}
|k\phi(q^{i+1})u_1-k\phi(q^i)u_1|>|q^{i+1}_1-q^i_1|\mbox{,}\\
|k\phi(q^{i+1})u_2-k\phi(q^i)u_2|>|q^{i+1}_2-q^i_2|\mbox{.}\\
\end{array}
\right.
\end{equation}

Note that (\ref{Lem.DPR.4.5.Eq.1}) can be required because $u_1$ and $u_2$ are non-zero. Further note that  $\phi(q^{i+1})$ is always distinct from $\phi(q^i)$ because $P$ is strongly flat. In the illustration shown in Fig. \ref{Fig.DPR.4.4}, $k$ is taken equal to $3$ and the segments with vertices $q^i$ and $\psi(q^i)$ are colored blue. Note that $q^1$ coincides with $\psi(q^1)$ because $\phi(q^1)=0$. Since $\psi$ is affine, $\psi(Q)$ is a lattice polygon whose vertices are exactly the images by $\psi$ of the vertices of $Q$. Denote $r^i=\psi(q^i)$. By (\ref{Lem.DPR.4.5.Eq.1}), the first coordinates and the second coordinates of $\psi(q^1)$ to $\psi(q^n)$ form two strictly monotone sequences. In other words, $\psi(Q)$ is oblique. By construction, the vertices of the convex hull of $P\cup\psi(Q)$ are exactly the vertices of $P$ and the vertices of $\psi(Q)$. The result then follows from Lemma \ref{Lem.DPR.1.BC}.
\end{proof}

Lattice pentagons have the following desirable property.

\begin{thm}\label{Thm.DPR.4.1}
Lattice pentagons are either flat or can be made flat by first inserting a single lattice point and then deleting a single vertex. 
\end{thm}
\begin{proof}
Consider a lattice pentagon $P$. Index the vertices of $P$ clockwise by $p^1$ to $p^5$. assume that $P$ is not flat. In this case, for any edge $E$ of $P$, the affine hulls of the two edges of $P$ adjacent to $E$ cannot be parallel and their intersection point $x$ is separated from the interior of $P$ by the affine hull of $E$. In other words, $x$ and the two extremities of $E$ are the vertices of a triangle contained in $H_E^-(P)$. This triangle will be denoted by $T^i$, where  $p^i$ is the vertex of $P$ opposite $E$. This situation is sketched on the left of Fig. \ref{Fig.DPR.4.5}, where the affine hulls of the edges of $P$ are shown as thin lines. The union of the cones $C_{p^i}(P)$, where $i$ ranges from $1$ to $5$, is the portion of $\mathbb{R}^2$ colored red in the figure. As can be seen, the union of the interiors of triangles $T^1$ to $T^5$ is precisely the set of the points of $\mathbb{R}^2$ that can be inserted in $P$. We will show that there is a lattice point in the interior of at least one of these five triangles.

Denote the edge of $P$ incident to $T^i$ by $E^i$. Also, call $F$ the line segment with vertices $p^2$ and $p^5$. We will review two cases. First assume that $F$ is not parallel to $E^1$. In this case, the points $p^3+p^5-p^2$ and $p^4+p^2-p^5$ cannot both belong to $H^-_{E^1}(P)$. By symmetry, we can assume that $p^3+p^5-p^2$ does not belong to $H^-_{E^1}(P)$. Call $x=p^3+p^5-p^2$. The resulting situation is depicted on the right of Fig. \ref{Fig.DPR.4.5}, at the top.
\begin{figure}[b]
\begin{center}
\includegraphics{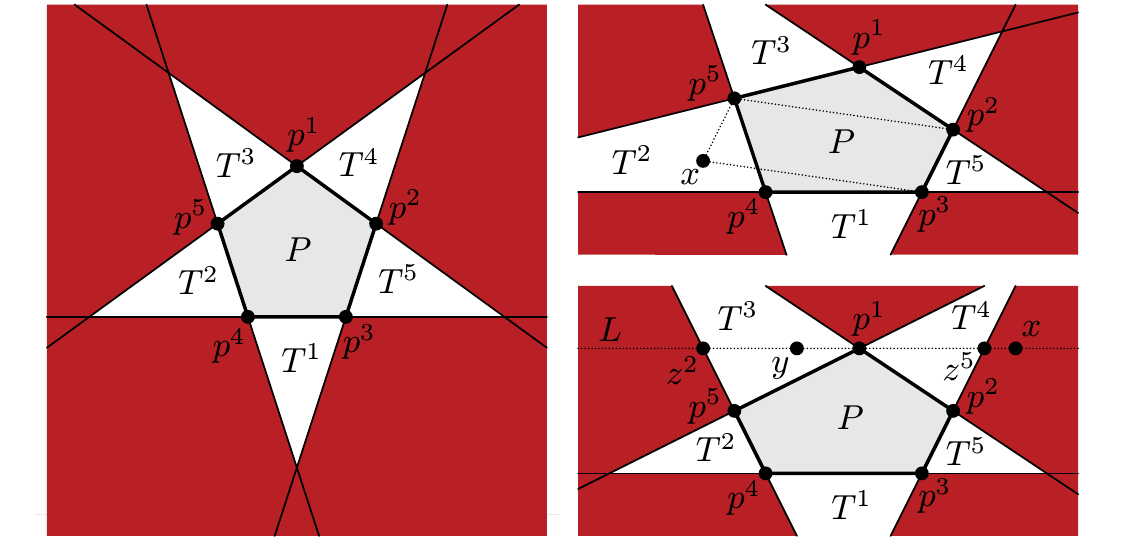}
\caption{The constructions in the proof of Theorem \ref{Thm.DPR.4.1}.}
\label{Fig.DPR.4.5}
\end{center}
\end{figure}
As shown in the figure, $x$ must then belong to the interior of $T^2$. Indeed, the quadrilateral with vertices $x$, $p^2$, $p^3$, $p^5$ is a parallelogram because two of its edges are translates of one another. Hence, $x$ belongs to the line through $p^5$ parallel to $E^5$. Since $P$ is not flat, $x$ cannot belong to $H^-_{E^3}(P)$, otherwise the lines parallel to $E^5$ through the vertices of $P$ would all intersect $E^4$. Moreover, $x$ cannot belong to $P$ either, otherwise the lines parallel to $E^5$ through the vertices of $P$ would all intersect $E^1$. In other words, $x$ must belong to the interior of $T^2$ as shown in the figure, and it can be inserted in $P$. In addition, $x$ is a lattice point as a linear combination of lattice points with integer coefficients. Call $Q$ the lattice pentagon obtained by first inserting $x$ in $P$ and by deleting $p^2$ from the resulting hexagon. As $P$ is not flat, the lines through its vertices parallel to $E^1$ (or to $E^3$) all intersect the segment that connects $p^1$ and $p^3$. Note that this segment is an edge of $Q$. Since $x$ belongs to $T^2$ the line through $x$ parallel to $E^1$ (or to $E^3$) also intersects this segment. As a consequence, $Q$ is a flat lattice pentagon.

Now assume that $F$ is parallel to $E^1$. This situation is depicted on the right of Fig. \ref{Fig.DPR.4.5}, at the bottom. Consider the point $x=p^1+p^3-p^4$. This point cannot belong to $H^-_{E^3}(P)$. Indeed, otherwise, the lines parallel to $E^1$ through the vertices of $P$ would all intersect $E^2$ and $P$ would be flat. The point $x$ cannot be contained in $P$ either because, then, the lines parallel to $E^1$ through the vertices of $P$ would all intersect $E^5$. Hence, if $x$ does not belong to $H^-_{E^5}(P)$, then it is contained the interior of $T^4$ and can be inserted in $P$. Again, $x$ is a lattice point as a linear combination of lattice points with integer coefficients. Therefore, inserting $x$ in $P$ and then deleting $p^4$ from the resulting hexagon transforms $P$ into a flat lattice pentagon, as desired. Now assume that $x$ belongs to $H^-_{E^5}(P)$. Observe that $x$ is contained in the line $L$ parallel to $E^1$ through $p^1$, sketched in Fig. \ref{Fig.DPR.4.5} as a dotted line. This line intersects the affine hulls of $E^2$ and $E^5$ in two points, which we denote by $z^2$ and $z^5$, respectively. Since $P$ is flat, the lines parallel to $E^5$ through the vertices of $P$ cannot all intersect $E^1$. As a consequence, the distance between $z^2$ and $z^5$ is greater than the distance between $p^2$ and $p^5$. In particular, the lattice point $x+p^5-p^2$, which we denote by $y$, is necessarily strictly between $x$ and $z^2$ in $L$. Since $P$ is flat, the distance between $p^3$ and $p^4$ is less than the distance between $p^2$ and $p^5$. Hence, $y$ lies strictly between $p^1$ and $z^2$ in $L$. This portion of $L$ is a subset of the interior of $T^3$ because $P$ is flat, and $y$ can therefore be inserted in $P$. As earlier, performing this insertion and, then deleting $p^3$ from the obtained hexagon transforms $P$ into a flat lattice pentagon, which completes the proof.
\end{proof}

Theorem \ref{thm.main.4}, that settles the second exception in the statement of Theorem~\ref{thm.main.3} can now be proven. Note that Theorems~\ref{thm.main.3}, \ref{thm.main.4}, and \ref{thm.Connect.lattice} collectively close the $2$-dimensional case of the question whether the subgraph induced in $\Lambda(d)$ by the polygons with $n$ or $n+1$ vertices is connected or not.

\begin{proof}[Proof of Theorem \ref{thm.main.4}]
By Theorem \ref{Thm.DPR.4.1}, there is a path from any lattice pentagon to a flat lattice pentagon in the subgraph induced in $\Lambda(2)$ by pentagons and hexagons. According to Lemma \ref{Lem.DPR.4.4}, any flat lattice pentagon can, in turn, be transformed into a strongly flat lattice pentagon within the same subgraph of $\Lambda(2)$. The desired result therefore follows from Lemma \ref{Lem.DPR.4.5}.
\end{proof}

\section{Discussion and open problems}\label{Sec.DPR.6}

We have introduced a graph structure on the $d$-dimensional polytopes contained in $\mathbb{R}^d$. We have proven, among other things, that this graph is connected, as well as its subgraph induced by lattice polytopes. The distances in this graph provide a measure of dissimilarity on polytopes in terms of how long it is to transform two of them into one another by a sequence of elementary moves. This allows to gather in a coherent metric structure very different objects from both the geometric and the combinatorial point of view.

This structure, and the results we obtained open up several new questions. For instance, recall that the subgraph induced in $\Gamma(d)$ by the polytopes with $n$ or $n+1$ vertices is always connected. 
We propose to investigate the subgraphs of $\Gamma(d)$ such that moves are allowed when a quantity other than the number of vertices is almost constant. In particular we ask the following.

\begin{qtn}\label{Qtn.DPR.1}
Consider a non-trivial interval $I\subset]0,+\infty[$. Is the subgraph induced in $\Gamma(d)$ by the polytopes whose volume belongs to $I$ connected?
\end{qtn}

Note that other measures than the volume of the polytope can be considered as well in Question \ref{Qtn.DPR.1}, and possibly several of them simultaneously (for instance, the volume and the number of vertices).

The main results in this article deal with lattice polytopes. These polytopes are often constrained to be contained in a hypercube \cite{AcketaZunic1995,DelPiaMichini2016,DezaManoussakisOnn2018,DezaPournin2018,KleinschmidtOnn1992,Naddef1989}. In this case, they form a nice (even if elusive) combinatorial class. The connectedness of $\Lambda(d,k)$ makes it possible to define a Markov chain on this combinatorial class whose stationary distribution is uniform \cite{DavidPourninRakotonarivo2018}. Some authors have considered lattice polytopes contained in a ball \cite{BaranyLarman1998} or in some arbitrary lattice polytope \cite{SoprunovSoprunova2016}. Pursuing this idea, we ask the following.

\begin{qtn}
For what balls $B$ is the subgraph induced in $\Lambda(d)$ by the polytopes contained in $B$ connected? For what lattice polytopes $P$ is the subgraph induced in $\Lambda(d)$ by the polytopes contained in $P$ connected?
\end{qtn}

Another graph that we have obtained results on is the subgraph induced in $\Lambda(d)$ by the polytopes with $n$ or $n+1$ vertices. The connectedness of this graph is particularly intriguing. For instance, when $d=2$, it follows from Theorem~\ref{thm.main.4}, Theorem~\ref{thm.Connect.lattice}, and Lemma \ref{Lem.DPR.4.1} that this graph is connected if and only if $n$ is equal to $3$ or $5$. When $d$ is greater than $2$, this graph is also sometimes connected and sometimes disconnected: Theorem~\ref{thm.Connect.lattice} tells that it is connected when $n=d+1$ for all $d\geq2$ and Theorem \ref{Thm.DPR.4.4} that it is diconnected for arbitrarily large values of $d$ and $n$. In particular, while we have completely settled the question in the $2$-dimensional case, the higher dimensional case is still open in general.

\begin{qtn}
Is the subgraph induced in $\Lambda(d)$ by the polytopes with $n$ or $n+1$ vertices sometimes connected when $d\geq3$ and $n\geq{d+2}$?
\end{qtn}

We ask two more questions on the structure of our graphs that are not directly related to the results we obtained in this article.

\begin{qtn}
Do the graphs $\Lambda(d,k)$, where either the value of $d$ or that of $k$ is fixed, form a family of expanders?
\end{qtn}

\begin{qtn}
What are the chromatic numbers of $\Gamma(d)$, $\Lambda(d)$, and $\Lambda(d,k)$?
\end{qtn}

\noindent{\bf Acknowledgements.} Lionel Pournin is partially supported by the the ANR project SoS (Structures on Surfaces), grant number ANR-17-CE40-0033.

\bibliography{MovesOnLatticePolytopes}

\providecommand{\MR}{\relax\ifhmode\unskip\space\fi MR }
\providecommand{\MRhref}[2]{%
  \href{http://www.ams.org/mathscinet-getitem?mr=#1}{#2}
}
\providecommand{\href}[2]{#2}
\begin{thebibliography}{10}

\bibitem{AcketaZunic1995}
Dragan Acketa and Jovi\v{s}a \v{Z}uni\'{c}, \textsl{On the maximal number of
  edges of convex digital polygons included into an $m\times{m}$-grid}, Journal
  of Combinatorial Theory A \textbf{69} (1995), 358--368.

\bibitem{BaranyKantor2000}
Imre B{\'a}r{\'a}ny and Jean-Michel Kantor, \textsl{On the number of lattice
  free polytopes}, European Journal of Combinatorics \textbf{21} (2000), no.~1,
  103--110.

\bibitem{BaranyLarman1998}
Imre B{\'a}r{\'a}ny and David~G. Larman, \textsl{The convex hull of the integer
  points in a large ball}, Mathematische Annalen \textbf{312} (1998), no.~1,
  167--181.

\bibitem{BaranyPach1992}
Imre B{\'a}r{\'a}ny and J{\'a}nos Pach, \textsl{On the number of convex lattice
  polygons}, Combinatorics, Probability and Computing \textbf{1} (1992), no.~4,
  295--302.

\bibitem{BarileBernardiBorisovKantor2011}
Margherita Barile, Dominique Bernardi, Alexander Borisov, and Jean-Michel
  Kantor, \textsl{On empty lattice simplices in dimension $4$}, Proceedings of
  the American Mathematical Society \textbf{139} (2011), no.~12, 4247--4253.

\bibitem{BlancoSantos2018}
M{\'o}nica Blanco and Francisco Santos, \textsl{Enumeration of lattice
  $3$-polytopes by their number of lattice points}, Discrete \& Computational
  Geometry \textbf{60} (2018), 756--800.

\bibitem{BlancoSantos2019}
M{\'o}nica Blanco and Francisco Santos, \textsl{Non-spanning lattice
  $3$-polytopes}, Journal of Combinatorial Theory A \textbf{161} (2019),
  112--133.

\bibitem{BogartHaaseHeringLorenzNillPaffenholzRoteSantosSchenck2015}
Tristram Bogart, Christian Haase, Milena Hering, Benjamin Lorenz, Benjamin
  Nill, Andreas Paffenholz, G{\"u}nter Rote, Francisco Santos, and Hal Schenck,
  \textsl{Finitely many smooth $d$-polytopes with $n$ lattice points}, Israel
  Journal of Mathematics \textbf{207} (2015), no.~1, 301--329.

\bibitem{BrionVergne1997}
Michel Brion and Mich{\`e}le Vergne, \textsl{Lattice points in simple
  polytopes}, Journal of the American Mathematical Society \textbf{10} (1997),
  no.~2, 371--392.

\bibitem{BrunsGubeladzeMichalek2016}
Winfried Bruns, Joseph Gubeladze and Mateusz Micha\l{}ek, \textsl{Quantum jumps
  of normal polytopes}, Discrete \& Computational Geometry \textbf{56} (2016),
  181--215.

\bibitem{DavidPourninRakotonarivo2018}
Julien David, Lionel Pournin and Rado Rakotonarivo, \textsl{A {Markov} chain
  for lattice polytopes}, Proceedings of GAScom 2018, CEUR Workshop
  Proceedings, vol. 2113, 2018, pp.~132--139.

\bibitem{DelPiaMichini2016}
Alberto {Del Pia} and Carla {Michini}, \textsl{On the diameter of lattice
  polytopes}, Discrete \& Computational Geometry \textbf{55} (2016), 681--687.

\bibitem{DezaManoussakisOnn2018}
Antoine Deza, George Manoussakis and Shmuel Onn, \textsl{Primitive zonotopes},
  Discrete \& Computational Geometry \textbf{60} (2018), 27--39.

\bibitem{DezaPournin2018}
Antoine Deza and Lionel Pournin, \textsl{Improved bounds on the diameter of
  lattice polytopes}, Acta Mathematica Hungarica \textbf{154} (2018), 457--469.

\bibitem{DezaOnn1995}
Michel Deza and Shmuel Onn, \textsl{Lattice-free polytopes and their diameter},
  Discrete \& Computational Geometry \textbf{13} (1995), 59--75.

\bibitem{DickensteinDiRoccoPiene2009}
Alicia Dickenstein, Sandra~Di Rocco and Ragni Piene, \textsl{Classifying smooth
  lattice polytopes via toric fibrations}, Advances in Mathematics \textbf{222}
  (2009), 240--254.

\bibitem{Gromov1981}
Mikhael Gromov, \textsl{Groups of polynomial growth and expanding maps},
  Publications Math{\'e}matiques de l'I.H.{\'E}.S. \textbf{53} (1981), 53--78.

\bibitem{HaaseZiegler2000}
Christian Haase and G{\"u}nter~M. Ziegler, \textsl{On the maximal width of
  empty lattice simplices}, European Journal of Combinatorics \textbf{21}
  (2000), no.~1, 111--119.

\bibitem{Kantor1999}
Jean-Michel Kantor, \textsl{On the width of lattice-free simplices}, Compositio
  Mathematica \textbf{118} (1999), no.~3, 235--241.

\bibitem{KleinschmidtOnn1992}
Peter Kleinschmidt and Shmuel Onn, \textsl{On the diameter of convex
  polytopes}, Discrete Mathematics \textbf{102} (1992), 75--77.

\bibitem{LagariasZiegler1991}
Jeffrey~C. Lagarias and G{\"u}nter~M. Ziegler, \textsl{Bounds for lattice
  polytopes containing a fixed number of interior points in a sublattice},
  Canadian Journal of Mathematics \textbf{43} (1991), 1022--1035.

\bibitem{Naddef1989}
Dennis Naddef, \textsl{The {H}irsch conjecture is true for $(0,1)$-polytopes},
  Mathematical Programming \textbf{45} (1989), 109--110.

\bibitem{NillZiegler2011}
Benjamin Nill and G{\"u}nter~M. Ziegler, \textsl{Projecting lattice polytopes
  without interior lattice points}, Mathematics of Operations Research
  \textbf{36} (2011), 462--467.

\bibitem{SantosValino2018}
Francisco Santos and {\'O}scar~Iglesias Vali{\~n}o, \textsl{Classification of
  empty lattice 4-simplices of width larger than two}, Transactions of the
  American Mathematical Society, to appear.

\bibitem{Sebo1999}
Andr{\'a}s Seb\H{o}, \textsl{An introduction to empty simplices}, Proceedings
  of IPCO 7, Lecture Notes in Computer Science, vol. 1610, 1999, pp.~400--414.

\bibitem{SoprunovSoprunova2016}
Ivan Soprunov and Jenya Soprunova, \textsl{Eventual quasi-linearity of the
  {M}inkowski length}, European Journal of Combinatorics \textbf{58} (2016),
  107--117.

\end{thebibliography}
\bibliographystyle{ijmart}

\end{document}